\documentclass[reqno,11pt,twoside]{amsart}
\usepackage{amsmath,amsfonts,amsfonts,amssymb}
\usepackage{amsbsy}
\usepackage{amsthm}
\usepackage{mathrsfs}
\usepackage{amscd}

\theoremstyle{break}
\newtheorem{lemma}{Lemma}

\newtheorem{proposition}[lemma]{Proposition}
\newtheorem{theorem}[lemma]{Theorem}
\newtheorem{corollary}[lemma]{Corollary}

\newtheorem{remark}[lemma]{Remark}
\newtheorem{definition}[lemma]{Definition}

\newcommand \QQ {{\mathbb Q}}
\newcommand \PP {{\mathbb P}}

\newcommand{\CC}{\ensuremath{\mathbb{C}}}
\newcommand{\NN}{\ensuremath{\mathbb{N}}}
\newcommand \FF {{\mathbb F}}
\newcommand \GG {{\mathbb G}}
\newcommand \AAA {{\mathbb A}}
\newcommand \RR {{\mathbb R}}
\newcommand \ZZ {{\mathbb Z}}

\newcommand \cE {{\mathcal E}}

\newcommand \cV {{\mathcal V}}

\newcommand{\q}{/\!\!/}

\newcommand{\coker}{\mathrm{coker\,}}

\newcommand{\Res}{\mathrm{Res}}
\newcommand{\gr}{\mathrm{gr}}

\newcommand{\To}{\longrightarrow}
\newcommand{\Xh}{\widehat{X}}

\title[Cohomology of graph hypersurfaces]{Framings for graph hypersurfaces}

\author{Francis Brown and Dzmitry Doryn}

\begin{document}

\begin{abstract}  We present a  method for computing the framing on the cohomology of graph hypersurfaces  defined by the Feynman differential form. This answers a question of Bloch, Esnault and Kreimer in the affirmative for an infinite class of graphs for which the framings 
are Tate motives. Applying this method to the modular graphs of Brown and Schnetz, we find that the Feynman differential form is not of Tate type in general. This finally disproves a folklore conjecture stating that the periods of Feynman integrals of primitive graphs in $\phi^4$ theory  \mbox{factorise} through a  category of mixed Tate motives.
  \end{abstract} 

\maketitle

\section{Introduction}
Let $G$ be a connected graph with $N_G$ edges. Its graph polynomial is defined  by associating a variable $\alpha_e$ to each edge $e$ of $G$, and setting
\begin{equation} \label{intro: PsiGdefinition}
\Psi_G=\sum_{T}\prod _{e\notin T} \alpha_e \quad \in \quad \ZZ[\alpha_1,\ldots,\alpha_{N_G}]
\end{equation}
where the sum runs over the  set of spanning trees $T$ of $G$. It is homogeneous of degree  equal to the number  $h_G$ of independent cycles in $G$.  The
graph hypersurface is  defined to be its zero locus  in projective space
$$X_G = \cV(\Psi_G) \subset \PP^{N_G-1}\ .$$
Following  \cite{Wein, BEK},  define the Feynman differential form to be 
\begin{equation} \label{introomegaGdef}
\omega_G = {\Omega_{N_G} \over \Psi_G^2} \in \Omega^{N_G-1} ( \PP^{N_G-1} \backslash X_G)
\end{equation}
where
$\Omega_{N_G} = \sum_{i=1}^{N_G} (-1)^i \alpha_i d \alpha_1 \wedge \ldots \widehat{d\alpha_i} \wedge \ldots d\alpha_{N_G}$, and let 
$\sigma$ be the coordinate simplex in real projective space $\sigma = \{(\alpha_1 : \ldots : \alpha_{N_G}): \alpha_i \geq 0\} \subset \PP^{N_G-1}(\RR)$.
When $G$ is primitive and overall logarithmically divergent (this means  $N_G=2h_G$ and $N_{\gamma}>2h_{\gamma}$ for all strict subgraphs $\gamma\subsetneq G$), the Feynman integral is 
\begin{equation} \label{introIG}
I_G = \int_{\sigma} {\omega_G}
\end{equation}
and is finite.  All known integrals $I_G$ are integral linear combinations of multiple zeta values \cite{BB, Census}.
The integral $(\ref{introIG})$ can be interpreted as the period of a  mixed Hodge structure $H$ which was defined in \cite{BEK},  called the graph motive. It is obtained by blowing up certain linear subspaces in $\PP^{N_G-1}$
and taking the relative cohomology of the complement of the (strict transform) of the graph hypersurface $X_G$. It is then relatively straightforward to show that the integration domain defines a class
\begin{equation} \label{introBettiframe}
[\sigma] \in \gr^W_0 H^{\vee}_B \cong \QQ(0)\ ,
\end{equation}
which is called  the Betti framing. The nature of the de Rham framing given by the relative cohomology class of the  integrand $ [\omega_G]$ is far from evident.

In \cite{BEK}, after their proof of $(\ref{introBettiframe})$, the authors write:
\\

\emph{`An optimist might hope for a bit more. Whether for all primitive divergent graphs, or for an identifiable subset of them, one would like that the maximal weight piece of $H_B$ should be Tate,
$$\gr^W_{max}  H_B  =\QQ(-p)^{\oplus r}$$
Further, one would like that there should be a rank one sub-Hodge structure $i:  \QQ(- p) \hookrightarrow
\gr^W_{max} H_B$ 	such that the image of $[\omega_G] \in  H_{dR}$ spans $i (\QQ(-p))$.'}
\\

One of the  main results of their paper is the following.

\begin{theorem} \cite{BEK}  Let $X_n$ denote the graph hypersurface for the wheel with $n$ spokes graphs, where $n\geq 3$. Then 
$$H^{2n-1} ( \PP^{2n-1} \backslash X_n) \cong \QQ(3-2n)$$
and $H^{2n-1}_{dR}( \PP^{2n-1} \backslash X_n)\cong \QQ[\omega_G]$ is spanned by  the Feynman differential form.
\end{theorem} 

The proof is an elaborate and ingenious argument which was generalised to the case of certain zig-zag graphs by the second author in his thesis \cite{Do}.

In this paper, we prove similar results for some infinite families of graphs by a rather different method. More precisely, 
for any  connected  graph $G$ which is called  denominator-reducible  (to be defined below), we show that 
$$\gr^W_{max} H^{N_G-1} ( \PP^{N_G-1} \backslash X_G)\cong \QQ(3-N_G)$$
and  indeed $\gr^W_{max} H_{dR}^{N_G-1} ( \PP^{N_G-1} \backslash X_G)\cong \QQ [\omega_G]$. 
It was shown in \cite{Br} that the class of denominator-reducible graphs contains the  wheel and  zig-zag families and all other graphs $G$ whose period $I_G$ is known.  The smallest non denominator-reducible graphs were studied in \cite{Br, BrSch} and have $h_G =8, N_G=16$. For one such  graph we prove that its de Rham framing is not of Tate type. This proves that the period $I_G$ cannot  factorize through a category of mixed
Tate motives.

A  corollary of our results is that if the maximum weight part of the graph cohomology complement is non-Tate, then it  must lie in  weight  lower than the generic weight $6-2N_G$. This suggests a remote 
possibility that the top generic  weight part of certain quantum field theories could still be mixed-Tate.

\subsection{Reduction of denominators and framings} The denominator reduction associated to an ordering on the edges of $G$ is a sequence of hypersurfaces
$$\cV(D_0) \subset \PP^{N_G-1}, \ldots,  \cV(D_k) \subset \PP^{N_G-k-1}$$
 defined as follows. The polynomial $D_0$ is by definition  $\Psi^2_G$, which is the denominator of $\omega_G$ defined in $(\ref{introomegaGdef})$. Let $\alpha_m$ denote the 
 variable corresponding to the $m^{\mathrm{th}}$ edge of $G$.  Suppose that $D_0, \ldots,  D_{m-1}$ are defined and non-zero.
 \begin{itemize}
\item \emph{(Generic step)} If  the  $(m-1)^{\mathrm{th}}$ denominator $D_{m-1}$ factorizes as a product
$$D_{m-1} = (f^m \alpha_m +f_m)(g^m \alpha_m + g_m)\ ,$$  where $f^m,f_m, g^m,g_m$ are polynomials which do not depend on $\alpha_m$,  and such that 
  $f^m g_m \neq g^m f_m$, then define
$$D_{m} = \pm (f^m g_m - g^m f_m) \ .$$ 
 \item  \emph{(Weight drop)} If  the  $(m-1)^{\mathrm{th}}$ denominator $D_{m-1}$ is a square
 $$D_{m-1} = (f^m \alpha_m +f_m)^2 \ ,$$  where $f^m,f_m,$ are polynomials which do not depend on $\alpha_m$,  define
$$D_{m} = \pm f^m f_m\ .$$
\end{itemize}
In all other cases, $D_m$ is not defined. One can show \cite{Br} that  the first five denominators 
$D_0,\ldots, D_5$ are always defined, and that a weight drop necessarily occurs at $m=1$ and $m=3$  (which explains  the generic weight of $6-2N_G$.)

If a  supplementary weight drop occurs after this point (for some $m\geq 4$), or if $D_m$ vanishes for some $m$, then $G$ is said to have \emph{weight drop}. 
We shall say that  a graph $G$ is \emph{denominator reducible} if there exists an ordering on its edges for which $D_m$ can be defined for all $m=0,\ldots, N_G-1$, i.e., every edge variable can be eliminated
by the simple procedure above. Clearly, the denominator reduction can be substantially generalized but this is not necessary for the present problem.

\begin{theorem} \label{introthmmain} Suppose that $G$  is connected, and satisfies  $N_G= 2h_G$, where $N_G \geq 5$.   Then for all $k\geq 3$ for which $D_k$ is defined, we have
$$ \gr^W_{2N_G-6} H^{N_G-1}  ( \PP^{N_G-1} \backslash X_G ) \cong \big(\gr^W_{2N_G-2k-2}  H^{N_G-k-1}  ( \PP^{N_G-k-1} \backslash D_k) \big)(2-k)\ .  $$
All higher weight-graded pieces of  $ H^{N_G-1}  ( \PP^{N_G-1} \backslash X_G )$ are zero.
\end{theorem}
The point of this theorem is that quite different graphs may have identical denominator reductions $D_k$ for some $k\geq k_0$.
A version of this, and the previous theorems, also  hold for the Hodge filtration.   The proof  uses  a cohomological Chevalley-Warning theorem due to Bloch, Esnault and Levine. 

\begin{corollary} Let $G$ be as in theorem \ref{introthmmain}. If $G$ is denominator-reducible then
$$ \gr^W_{2N_G-6} H^{N_G-1}  ( \PP^{N_G-1} \backslash X_G ) \cong \QQ(3-N_G) \ . $$
Furthermore,  $ \gr^W_{2N_G-6} H_{dR}^{N_G-1}  ( \PP^{N_G-1} \backslash X_G )$ is spanned by $[\omega_G]$.
\end{corollary}
Thus for denominator-reducible graphs, the Feynman differential form provides a Tate framing for the maximal weight piece of the de Rham cohomology.
Keeping track of the framings requires a  different argument from the proof of theorem \ref{introthmmain}.
\begin{corollary} Let $G$ be as in theorem  \ref{introthmmain}.  If $G$ has weight-drop then
$$H^{N_G-1}  ( \PP^{N_G-1} \backslash X_G ) \hbox{  has weights }   < 2N_G-6\ .$$
\end{corollary}

Various combinatorial criteria for a graph $G$ to have weight drop were established in \cite{BY} and \cite{BSY}. Combining these  criteria with  theorem \ref{introthmmain} proves upper bounds on the Hodge-theoretic weights of the 
period $(\ref{introIG})$.

\subsection{Non-Tate counterexamples} 
The  smallest graphs in $\phi^4$ theory ($G$ is said to be in $\phi^4$ theory if all its  vertices have degree  at most 4) which are not denominator-reducible  occur at 8 loops.  One of them was studied  in \cite{BrSch}.

\begin{theorem} Let $G_8$ be the 8-loop modular graph of \cite{BrSch}.  Then 
$$\gr^{13,11} H_{dR}^{15}(\PP^{15} \backslash X_{G_8}) \hbox{ is 1-dimensional, spanned by the class of }  [\omega_{G_8}]\ .$$  
\end{theorem}

To put this result in  context,  it was proved in \cite{BrSch}  that the point-counting function of the corresponding  graph hypersurface over a finite field $\FF_q$ with $q= p^n$ elements satisfies
$$| X_{G_8} (\FF_q) | \equiv -a_q q^2 \pmod{pq^2}$$
where the integers $a_q$ are Fourier coefficients of a certain modular form of weight $3$. One can deduce that the map $q\mapsto | X_{G_8} (\FF_q) |$ is  not a polynomial (or quasi-polynomial) function of $q$, and show  that the Euler characteristic of $H_c(X_{G_8})$ is not mixed-Tate.  
We do not wish to repeat a lengthy history of the point-counting problem  for graph hypersurfaces here: the interested reader can refer to  the summaries in \cite{MaF}, \cite{BrSch}, \cite{AM} or the papers \cite{BB}, \cite{Stem} 
for further information.

However, the possibility remained that the Feynman period $I_{G_8}$ could be supported on a smaller part  of the  cohomology  (or graph motive) which is in fact mixed-Tate. The previous theorem rules out this possibility: 
the non-Tate contribution to the cohomology arises precisely because of the Feynman differential form.
As a result, this disproves a folklore conjecture (mentioned, for example, in  \cite{MaE} \S1.6) that the periods of Feynman graphs factor through a category of mixed Tate motives.

The reader may have noticed that the above theorems pertain to the absolute cohomology of the graph hypersurface complement, and not the full graph motive $H$. We expect that the methods of the present paper, together
with some standard spectral sequence arguments for relative cohomology will be enough to deduce corresponding statements for the full graph motive $H$.
\\

\emph{Acknowledgements}.  The first named author is partially supported by ERC grant 257638, and wishes to thank Humboldt university for hospitality and support.
The second named author was supported by the same grant in 2011 and is currently supported by Humboldt university.

\section{Preliminaries}

We first gather some preliminary results on graph polynomials and some basic identities for them. Secondly, for the benefit of physicists, we review
some well-known exact sequences for cohomology. 

\subsection{Polynomials related to graphs} The graph polynomial $\Psi_G$ can be written as a determinant as follows.
Let $G$ be a connected graph without tadpoles (self-loops), and let $E$ denote its set of edges, and $V$ its set of vertices.
Choose an orientation of its edges. For
an edge $e$ and vertex $v$ set $\varepsilon_{e,v}$ to be 1 if $v$ is the
source of $e$, -1 if $v$ is the target, and 0 otherwise, and let  $\cE_G$ be the
$|E|\times (|V|-1)$ matrix obtained by deleting one of the columns of
$(\varepsilon)_{e,v}$. 

Consider the $(|E|+|V|-1)\times (|E|+|V|-1)$ matrix $M_G$
$$
M_G=\left( \begin{array}{c|c}
A & \cE_G\\
\hline
-\cE_G^T & 0
\end{array} \right),
$$
where $A$ is the diagonal matrix with entries $\alpha_e$, $e\in E$. Write $N:=|E|$.
One can show by the matrix-tree theorem that the graph polynomial  $(\ref{intro: PsiGdefinition})$   associated to the graph $G$ is simply given by  the determinant of $M_G$
\begin{equation}
\Psi_G=  \det(M_G)Ê\ .
\end{equation}
In order to  understand the structure of graph hypersurfaces we require various  identities involving  some  other polynomials based on the matrix $M_G$.
\begin{definition}
Let $I,J,K$ be subsets of $E$ which satisfy $|I|=|J|$. Let $M_G(I,J)_K$ denote
the matrix obtained from $M_G$ by removing the rows (resp. columns) indexed by
the set $I$(resp. $J$) and setting $\alpha_e=0$ for all $e\in K$. Define the
Dodgson polynomial to be the corresponding minor
\begin{equation}
\Psi^{I,J}_{G,K}:=\det M_G(I,J)_K.
\end{equation}
Strictly speaking, the polynomials $\Psi^{I,J}_{G,K}$ are defined up to a sign which depends on the choices involved in defining $M_G$. A simple-minded way to fix the signs is to fix
a matrix $M_G$ once and for all for any given graph $G$.
\end{definition}
The Dodgson polynomials satisfy many identities, which can be found
in \cite{Br}. We only recall two of them here.

\begin{enumerate}
\item The contraction-deletion formula. For any $\alpha_e$, $e\in E$
\begin{equation}
\Psi^{I,J}_{G,K}=\Psi^{Ie,Je}_{G,K}\alpha_e+\Psi^{I,J}_{G,Ke}
\end{equation}
where  $\Psi^{Ie,Je}_{G,K} = \pm   \Psi^{I,J}_{G\backslash e,K}$ and  $\Psi^{I,J}_{G,Ke}=\pm \Psi^{I,J}_{G\q e,K}$. Here, $G\backslash e$ (respectively, $G\q e $) denotes the graph obtained by
deleting (contracting) the edge $e$.
\item The  Dodgson identity. Let $a,b,x\notin I\cup J\cup K$. Then
\begin{equation}
\Psi^{Ix,Jx}_{G,K} \Psi^{Ia,Jb}_{G,Kx} -  \Psi^{I,J}_{G,Kx} \Psi^{Iax,Jbx}_{G,K} =\Psi^{Ix,Jb}_{G,K}\Psi^{Ia,Jx}_{G,K} \ .
\end{equation}
\end{enumerate}
We will  mostly deal  with  graphs which have a 3-valent vertex (this holds for all non-trivial physical graphs). 
 For such graphs it is convenient to use the following notation (see  \cite{Br}, Example 32).
Suppose that  $G$  has a  3-valent vertex adjoined to   edges $e_1$, $e_2$, and $e_3$. Define
\begin{equation}\label{b9}
f_0:=\Psi_{G\backslash\{1,2\}\q 3},\;\; f_i:=\Psi^{j,k}_{G,i},\;\; f_{123}=\Psi_{G\q\{1,2,3\}} , 
\end{equation}
with $\{i,j,k\}=\{1,2,3\}$. Note also that $f_0 = \Psi^{ij,jk}_G$.  The structure of $\Psi_G$ is
\begin{equation}\label{Psistar}
\Psi_G=f_0(\alpha_1\alpha_2+\alpha_1\alpha_3+\alpha_2\alpha_3)+(f_1+f_2)\alpha_3+(f_2+f_3)\alpha_1+(f_1+f_3)\alpha_2+f_{123},
\end{equation}
where the $f_I$'s are related by the identity
\begin{equation}\label{b11}
f_0f_{123}=f_1f_2+f_1f_3+f_2f_3.
\end{equation}

\subsection{Cohomology and  exact sequences}

Throughout this paper we shall work over a field $k$ of characteristic zero.  Let $X \subset\PP^N$ be a quasi-projective  but not necessarily smooth scheme defined over $k$.

Recall that  the Betti cohomology $H^n(X)=H^n(X; \QQ)$ has a  $\QQ$-mixed Hodge structure  (\cite{De2},   2.3.8). This consists of  a finite increasing filtration 
$W_{\bullet} H^n(X)$ called the weight, and a finite decreasing filtration $F^{\bullet} H^n (X)\otimes\CC$ called the Hodge filtration, such that the induced filtration $F$ on the associated graded pieces $\gr^W_k$
is a pure Hodge  structure of weight $k$.  The category of mixed Hodge structures (MHS) is an abelian category.
We shall frequently use the fact that 
\begin{equation}
H \mapsto \gr^W_k H \qquad \hbox{ (resp. }  H \mapsto \gr^k_F H\otimes \CC \hbox{ )}
\end{equation}
is  an exact functor from the category of mixed Hodge structures to the category of pure Hodge structures over $\QQ$ (respectively, to the category of 
graded complex vector spaces, the grading being given by the weight). We shall often write $\gr(k)$ when we wish to consider both $\gr^k_F$
or $\gr^W_{2k}, \gr^W_{2k+1} $ simultaneously; the extension of scalars should be clear from the context. 
The Hodge numbers $h^{p,q}$ are defined by  $h^{p,q} = \dim_{\CC} \gr_F^p \gr^W_{p+q} H\otimes \CC$. 
They satisfy (\cite{De2} 3.2.15, \cite{De3}  8.2.4),
\begin{eqnarray} \label{Xhodgenos}
&& h^{p,q} (H^n(X) )\neq 0 \quad \Rightarrow \quad 0 \leq p,q  \leq n \ ,  \\
&X \hbox{ smooth} & h^{p,q} (H^n(X) )\neq 0  \quad \Rightarrow  \quad p+q \geq n\  ,  \nonumber \\
&X \hbox{ proper} & h^{p,q} (H^n(X) )\neq 0  \quad \Rightarrow  \quad p+q \leq n\ . \nonumber 
\end{eqnarray}
and are symmetric: $h^{p,q}=h^{q,p}$ for all $p,q$.
There is a unique pure Hodge structure of dimension one in  each even  weight $2n$, which is the pure Tate Hodge structure denoted by $\QQ(-n)$. Its  Hodge numbers satisfy  $h^{p,q}=1$ if $(p,q)=(n,n)$, and are zero
otherwise. More generally, a  mixed Hodge structure $H$ is said to be 
\emph{mixed Tate} if 
\begin{equation}
h^{p,q} \neq 0\quad  \Rightarrow \quad  p = q \ . 
\end{equation}
Equivalently,  $\gr^W_{2k} H \cong \QQ(-k)^{\oplus n_k}$, where $n_k \in \NN$, and $\gr^W_{2k+1} H =0$ for every $k$.
Recall the notation for Tate twists $H(n) = H \otimes \QQ(n)$, which shifts the weight filtration by $-2n$, and the Hodge filtration by $n$.

We will also make use of cohomology with compact support $H^k_c(X)$, which  carries  a mixed Hodge structure, and  
is functorial with respect to proper morphisms. When $X$ is proper,  $H^k_c(X) = H^k(X)$, 
and when $X$ is smooth  of equidimension $n$, there is a canonical isomorphism:
 \begin{equation} 
 H_c^k(X)^{\vee} \cong H^{2n-k}(X) (n)
 \end{equation}
 where the superscript $\vee$ denotes the dual mixed Hodge structure.

For a proper scheme $X\subset\PP^N$ and $r<2N$ we define the primitive cohomology
\begin{equation} H^r_{prim}(X):=\coker (H^r(\PP^{N})\rightarrow
H^r(X)).
\end{equation} 
Finally, we will also consider  de Rham cohomology $H^{n}_{dR}(X; k)$. When $X$ is smooth and affine it is computed by the cohomology of the complex of
regular  forms  $\Omega^{\bullet}(X;k)$.
Typically, $X$ will be defined over $k=\QQ$, and we shall simply write
$H^n_{dR}(X)$ for $H^n_{dR}(X;\QQ)$. The Betti-de Rham comparison gives  an isomorphism
\begin{equation}\label{BettideRham}
H^n(X) \otimes \CC \cong H^n_{dR} (X)\otimes \CC\ .
\end{equation}

\subsubsection{Exact sequences and notations} Our main results concern the computation of the cohomology  of graph hypersurfaces, or their open complements, in middle degree.
The arguments involve applying various standard exact sequences many times over. For the convenience of the reader, we state  them below.

Let $X$ be a proper  scheme of the type considered above,  and let  $Z$ be a closed subscheme. Write  $U=X\backslash Z$. Then there is an exact sequence 
\begin{equation} \label{seq : loc1}
\longrightarrow H_c^r(U;\QQ)\longrightarrow H^r(X;\QQ)\longrightarrow H^r(Z;\QQ)\longrightarrow H^{r+1}_c(U;\QQ) \longrightarrow 
\end{equation}
which is called the \textit{localization sequence}. Since our coefficients are always in $\QQ$, we shall omit them hereafter. The sequence remains valid after replacing cohomology with   primitive cohomology, i.e., after adding  a subscript $prim$ to $ H^r(X)$ and $H^r(Z)$.  With the same hypotheses,  we can also consider cohomology with support $H_{Z}^r(X)$ (\cite{Spa}), which sits inside another localization sequence
\begin{equation} \label{seq : loc2}
\longrightarrow H_{Z}^r(X)\longrightarrow H^r(X)\longrightarrow H^r(U)\longrightarrow H_{Z}^{r+1}(X) \longrightarrow
\end{equation}
In the case when $Z$ is smooth in $X$, one has a \textit{Gysin isomorphism} 
\begin{equation}
H^r_Z(X)\cong H^{r-2}(Z)(-1).
\end{equation} 
Combining the last two gives  the \textit{Gysin sequence} 
\begin{equation} \label{seq :  Gysin}
\longrightarrow H^{r-1}(U)\longrightarrow H^{r-2}(Z)(-1)\longrightarrow H^r(X)\longrightarrow H^r(U)\longrightarrow
\end{equation}
In de Rham cohomology, the map $H_{dR}^r(U) \rightarrow H_{dR}^{r-1}(Z)(-1)$ in the previous sequence is given by  the residue map.
Finally, when $X$ admits a closed  covering $X=X_1\bigcup X_2$, we have the  \emph{Mayer-Vietoris sequence}
\begin{equation} \label{seq :  MV}
\longrightarrow H^r(X)\longrightarrow H^r(X_1)\oplus H^r(X_2)\longrightarrow H^r(X_1\cap X_2)\longrightarrow  H^{r+1}(X) \longrightarrow 
\end{equation}
The above sequences are motivic in the sense that they correspond to distinguished triangles  in 
a triangulated category of mixed motives over $k$ \cite{Voe}. 
 In particular, the sequences are valid in a suitable abelian category of motives when it exists (such as a category of mixed Tate motives
over a number field).

\begin{remark} Artin vanishing states that  if $U$ is smooth and  affine of finite type over $k$, then we have
$H^r(U)=0$ for $r>\dim U$  (\cite{SGA7}, XIV). 
 Since Artin vanishing is not presently known to be motivic, we  preferred to  avoid  using it at all costs, even though  it could have marginally simplified certain arguments.
 \end{remark}

We shall frequently use the following remark (compare [BEK], lemma 11.4).
\begin{lemma} \label{lemcone} Let $\ell\geq m$ and let $V\subset \PP^\ell$, $W\subset \PP^m$ be hypersurfaces such that $V$ is   a cone over $W$.
Then $H_{prim}^n(V) \cong  H^{n-2(\ell-m)}_{prim} ( W) (m-\ell).$
\end{lemma}
\begin{proof}   The projection $\PP^{\ell} \backslash V \rightarrow \PP^{m} \backslash W$ is  an $\AAA^{\ell-m}$-fibration. By homotopy invariance, 
$H_c^n(\PP^{\ell} \backslash V) \cong H_c^{n-2(\ell-m)} (\PP^{\ell} \backslash W)(m-\ell).$ A localisation sequence for $V\subset \PP^{\ell}$ gives
$H^n_{prim}(V) \cong H_c^n(\PP^{\ell} \backslash V)$, and similarly for $W$, which gives the statement.
\end{proof}

Throughout the paper  we write $\cV(f_1,\ldots,f_m)$  for the
vanishing locus of homogeneous polynomials $f_1,\ldots,f_m \in k[x_1,\ldots, x_N]$ 
in $\PP^{N-1}$.

\subsection{Chevalley-Warning theorem for cohomology}
The Chevalley-Warning theorem states that  if  $X \subset \PP^N$ is a hypersurface of degree $d\leq N$ defined over a finite field $\FF_q$ with $q$ elements,  then one has
$$| X(\FF_q) |  \equiv | \PP^N(\FF_q)| \mod q\ .$$
The following cohomological version of this theorem was proved by Bloch, Esnault and Levine in \cite{BEL}, 
using  a beautiful geometric  idea due to Roitman.  It will play a crucial role in the rest of this  paper.

\begin{theorem}\label{t20}
Let $X\subset\PP^N$ be a hypersurface of degree $d\leq N$ over a field $k$ of characteristic zero.  Then 
$$gr^0_F
H^{n}(X)=0\  , \hbox{  for all  } n \geq 1 \ . $$
\end{theorem}
This implies that the Hodge numbers of $X$ satisfy $h^{0,q}=0$ for all $q\geq 1$, and by symmetry of the Hodge numbers, we also obtain $h^{p,0}=0$ for all $p\geq 1$. In particular, 
$h^{i,j}=0$ for $i+ j = 1$, and $\gr^W_1 H^n(X)=0$ for all $n\geq 1$. Thus   $$gr^W_i
H^{n}_{prim}(X,\QQ)=0$$ for $i<2$ and all $n$.
Recall that we shall write this $\gr(0) H^{n}_{prim}(X,\QQ)=0$ to denote both statements for the Hodge and weight filtrations simultaneously.

The  result easily extends to complete intersections.
\begin{theorem}\label{t24}
Let $X \subset\PP^N$ be the intersection of $r$ hypersurfaces $X_1,  \ldots, X_r$ where $X_i$ is of degree $d_i$, and $d_1+\ldots+d_r\leq N$. Then
 $$\gr(0) H^{n}_{prim}(X)=0 \hbox{  for all  }  n\ . $$
\end{theorem}
\begin{proof}
We prove a  stronger statement, namely, if $X= \cap_{i=1}^r X_i $ and $Z_1,\ldots, Z_s$ are the set of irreducible components of $\cup_{i=1}^r X_i$ and satisfy  
$\sum_{i=1}^s \deg Z_i \leq N$, then we have 
$\gr(0) H^{n}_{prim}(X)=0$  for all  $n$.
The proof is by induction. For $r=1$ it is the previous theorem. Let $Y= X_1 \cap \ldots \cap X_{r-1}$.  A Mayer-Vietoris  sequence gives
\begin{equation}
\rightarrow H^n_{prim}(Y)\oplus H^n_{prim} (X_r)\rightarrow H^n_{prim}(X)\rightarrow H^{n+1}_{prim}(Y\cup X_r) \rightarrow
\end{equation}
By induction hypothesis,  the vanishing  statement holds for the summands on the  left. It  holds for the  term on the right since 
$Y \cup X_r = (X_1 \cup X_r) \cap \ldots \cap( X_{r-1} \cup X_r)$, and $\cup_{i=1}^{r-1} (X_i \cup  X_r) $ has irreducible components $Z_1,\ldots, Z_s$. The result follows from the exactness of $\gr(0)$.
\end{proof}
We will mostly apply this theorem in the case  $r=2$.

\newpage
\section{Cohomological denominator reduction}

For any homogeneous polynomials  $f=f^1 x + f_1$ and $g=g^1 x+g_1$, where $f^1,f_1,g^1,g_1\in k[x_2,\ldots,x_N]$,  let us denote their resultant by:
\begin{equation}
[f,g]_x=f^1g_1-f_1g^1.
\end{equation}

\subsection{The generic reduction  step}
 
  Let  $f, g$ be  polynomials as above, satisfying $\deg fg \leq  N$.  Suppose that their resultant has a  factorization
 \begin{equation} \label{fgfactorizesasab}
  [f,g]_x= ab
  \end{equation}
 where $a,b$ are polynomials of  degree $\geq 1$.  Then the following  holds.

\begin{proposition}  \label{propDenomRed} (Denominator reduction) With $f,g,a,b$ as  above, 
\begin{equation}\label{DenomRedGenericStep}
\gr(0) H_{prim}^n( \cV(f,g) )  \cong \gr(0) H_{prim}^{n-1}( \cV(a,b) )\ ,
\end{equation}Ê
for all $n$,  where $\cV(f,g) \subset \PP^{N-1}$ and $\cV(a,b) \subset \PP^{N-2}$.
 \end{proposition}
It is sometimes more convenient to state this in the form
\begin{equation}\label{DenomGenericStepVersion2}
\gr(0) H_{prim}^n( \cV(fg) )  \cong \gr(0) H_{prim}^{n-1}( \cV(ab) )\ , \hbox{ for all } n \ .
\end{equation}Ê

The proof is split into two parts.
\begin{proposition}\label{p5}
Let  $f=f^1 x + f_1$ and $g=g^1 x+g_1$ be homogeneous polynomials, where $f^1,f_1,g^1,g_1\in k[x_2,\ldots,x_N]$,  and $[f,g]_x \neq 0$. Suppose that, for all $n$, \begin{equation} \label{p5assumption}
\gr(0) H^n_{prim}(\cV(f^1,g^1))=0\ .
\end{equation}Then, for all $n$, 
\begin{equation}
\gr(0) H^{n}_{prim}(\cV(f ,g )) =  \gr(0)
H^n_{prim}(\cV(f^1g_1-f_1g^1))\ .
\end{equation}
\end{proposition}
\begin{proof}
Let us write $\PP^{N-1}= \PP(x:x_2:x_3:\ldots:x_N)$, $\PP^{N-2}=\PP(x_2:\ldots:x_N)$, and denote
$\cV(f,g)\subset \PP^{N-1}$ simply by $R$.

The closed subscheme $R\cap \cV(f^1,g^1)\subset
R$ gives rise to a sequence $(\ref{seq : loc1})$:
\begin{multline}\label{b37}
\rightarrow H^{n-1}_{prim}(R\cap\cV(f^1,g^1)) \rightarrow
H^n_c(R\backslash
(R\cap\cV(f^1,g^1)))\rightarrow\\
H^n_{prim}(R) \rightarrow H^n_{prim}(R\cap\cV(f^1,g^1)) \rightarrow 
\end{multline}
It follows from the linearity of $f$ and $g$ that the intersection $R \cap \cV(f^1,g^1)$ is  a cone over $\cV(f^1,f_1,g^1,g_1)$, and so lemma \ref{lemcone} implies that
\begin{equation}
H^{m}_{prim}(R\cap\cV(f^1,g^1))\cong
H^{m-2}_{prim}(\cV(f^1,f_1,g^1,g_1))(-1)
\end{equation}
for any $m$. Since  the  grading functor is exact,  (\ref{b37})  implies that
\begin{equation} \label{p5eq1}
gr(0) H^n_{prim}(R) \cong gr(0) H^n_c(R\backslash
(R\cap\cV(f^1,g^1))) \ .
\end{equation}
Now the projection $\PP^{N-1} \rightarrow \PP^{N-2}$ from the point $p= (1: 0 : \ldots :0)$ gives an isomorphism from 
$R\backslash
R\cap\cV(f^1,g^1)$   to
$\cV(f^1g_1-f_1g^1)\backslash\cV(f^1,g^1)$ (the  inverse map  is given by  $x=-f_1/f^1$ on the complement of $\cV(f^1)$  and by
$x=-g_1/g^1$ on the complement of $\cV(g^1)$).  Therefore we have:
\begin{equation}\label{p5eq2}
H^n_c(R\backslash (R\cap\cV(f^1,g^1)))\cong
H^n_c(\cV(f^1g_1-f_1g^1)\backslash\cV(f^1,g^1)).
\end{equation}
A final application of the localization sequence $(\ref{seq : loc1})$ for the inclusion of the closed subscheme
$\cV(f^1,g^1)\subset\cV(f^1g_1-f_1g^1)$
gives \begin{multline} \label{p5mainseq}
\rightarrow H^{n-1}_{prim}(\cV(f^1,g^1))\rightarrow
H^n_c(\cV(f^1g_1-f_1g^1)\backslash \cV(f^1,g^1))\\\rightarrow
H^n_{prim}(\cV(f^1g_1-f_1g^1))\rightarrow
H^n_{prim}(\cV(f^1,g^1))\rightarrow
\end{multline}
By  assumption $(\ref{p5assumption})$, the functor  $\gr(0)$ induces an isomorphism between the graded pieces of the two terms in the middle of the previous sequence.
Combining this with isomorphisms $(\ref{p5eq1})$ and $(\ref{p5eq2})$,  we conclude that, for all $n$, 
\begin{multline} \nonumber 
\gr(0) H^n_{prim}(R) \overset{(\ref{p5eq1})}{\cong}
\gr(0) H^n_c(R\backslash (R\cap\cV(f^1,g^1))) \overset{(\ref{p5eq2})}{\cong} \\
\gr(0) H^n_c(\cV(f^1g_1-f_1g^1)\backslash\cV(f^1,g^1))\cong \gr(0)
H^n_{prim}(\cV(f^1g_1-f_1g^1)) \ .
\end{multline}
\end{proof}
 
\begin{lemma} Let $a, b \in k[x_1,\ldots, x_N]$ be homogeneous  polynomials which satisfy $\deg a, \deg b < N$. Then for all $n$, 
\begin{equation}\label{b51} 
    \gr(0)H_{prim}^n(\cV(ab))\cong \gr(0)H_{prim}^{n-1}(\cV(a,b))\ .
\end{equation}
\end{lemma}

\begin{proof}
The Mayer-Vietoris sequence $(\ref{seq :  MV})$ gives
\begin{equation}\label{b50}
    \longrightarrow H_{prim}^{n-1}(\cV(a,b))\longrightarrow H_{prim}^n(\cV(ab))\longrightarrow H_{prim}^n(\cV(a))\oplus H_{prim}^n(\cV(b))\longrightarrow
\end{equation}
By the assumption on the degrees, the  Chevalley-Warning theorem \ref{t20} 
 implies that   $\gr(0)H_{prim}^m(\cV(a))$ and $\gr(0)
H_{prim}^m(\cV(b))$  vanish for all $m$.    The  previous sequence then  gives the required isomorphism
$$\gr(0)H_{prim}^n(\cV(ab))\cong \gr(0)H_{prim}^{n-1}(\cV(a,b))\ .$$ 
\end{proof}
Now we return to the proof of proposition \ref{propDenomRed}.

 \begin{proof} By the assumption on the degrees, $\deg f^1g^1 \leq  N-2$, and therefore $\cV(f^1,g^1)$ satisfies the condition of the Chevalley-Warning theorem \ref{t24}, and
 so $(\ref{p5assumption})$ holds. By proposition $\ref{p5}$, we have
 $$ \gr(0) H_{prim}^n( \cV(f,g) )  \cong \gr(0) H_{prim}^n( f^1g_1 - f_1g^1 )  =\gr(0) H_{prim}^n( \cV(ab) ) $$
  Since the factorization $(\ref{fgfactorizesasab})$ is non-trivial, $\deg a,  \deg b \leq N-2$ and the previous lemma implies that 
  $\gr(0) H_{prim}^n( \cV(ab) )  \cong \gr(0) H_{prim}^{n-1}( \cV(a,b) ) $, as required.
   \end{proof}

   \subsection{Cohomological vanishing}
   The argument in the proof of proposition  \ref{p5} can be turned around, giving the following lemma.
   
   \begin{lemma}  \label{lemabtoacommab}
   Let  $f=f^1 x + f_1$ and $g=g^1 x+g_1$ be homogeneous polynomials, where $f^1,f_1,g^1,g_1\in k[x_2,\ldots,x_N]$,  and $[f,g]_x \neq 0$. 
    Suppose for all $n$ that 
   \begin{equation} \label{assumpresvanishes}  \gr(0) H^n_{prim} (\cV(f^1g_1-f_1g^1)) =0\ .
   \end{equation}
   Then for all $n$,  
   \begin{equation} \label{fgloosea1} 
   \gr(0) H^n_{prim} ( \cV(f,g)) = \gr(0) H^n_{prim} ( \cV(f^1,g^1)) \ .
   \end{equation} 
   \end{lemma}
      \begin{proof}
   Equation $(\ref{fgloosea1})$ follows immediately  on applying the grading functor to  (\ref{p5mainseq}), and using isomorphisms (\ref{p5eq2}) and  (\ref{p5eq1}),  and  assumption $(\ref{assumpresvanishes})$.
\end{proof}
   
   Note that if $G$ is a connected graph with at least 3 vertices, then it follows from Euler's formula that $h_G\leq N_G-2$.
  The next proposition is the cohomological version of an analogous statement  in the Grothendieck ring which was proved in \cite{BrSch}. Equation $(\ref{PsiGVanishestoorder2})$ 
    will be reproved in the next section  under the more restrictive hypothesis that $G$ has a 3-valent vertex.
    
\begin{proposition} Let $G$ be a connected graph  with at least 3 vertices.
Then   \begin{equation}  \label{PsiGVanishestoorder2}
     \gr(i) H^n_{prim} ( \cV(\Psi_{G}))=0 \hbox{  for } \, i=0,1 \hbox{ and all } n \ , 
     \end{equation} 
  and  for any edge $e \in G$, 
\begin{equation}  \label{PsiGupperlowerVanishes}
    \gr(0) H^n_{prim} ( \cV(\Psi^e_{G}, \Psi_{G,e}))=0  \hbox{ for all } n \ . 
 \end{equation}
\end{proposition}

\begin{proof} We first prove $(\ref{PsiGupperlowerVanishes})$ by induction on the number of edges of $G$. The induction step is as follows.
Let $e'$ be an edge of $G$ distinct from $e$. By the contraction-deletion relations,
$\Psi^{e}_G= \Psi^{e e'}_G \alpha_{e'}+   \Psi^{e}_{G,e'}$  and $\Psi_{G,e}= \Psi^{e'}_{G,e} \alpha_{e'} +   \Psi_{G,e e'}$. 
We wish to apply the previous lemma with $f= \Psi^{e}_G$ and $g=\Psi_{G,e}$, and $x= \alpha_{e'}$.
The Dodgson identity implies that the resultant of $f$ and $g$ factorizes:
  $$\Psi^{e}_{G,e'} \Psi^{e'}_{G,e}- \Psi^{ee'}_{G}\Psi_{G,ee'} = (\Psi^{e,e'}_G)^2\ ,$$ and  in particular, 
$$\gr(0) H^n_{prim}(\cV( \Psi^{e}_{G,e'} \Psi^{e'}_{G,e}- \Psi^{ee'}_{G}\Psi_{G,ee'})) = \gr(0) H^n_{prim} ( \cV(\Psi^{e,e'}_G))$$
for all $n$.  The polynomial $\Psi^{e,e'}_G$ is of degree $h_{G}-1$ and so by the Chevalley-Warning theorem \ref{t20}, the right-hand side of the previous equation vanishes, and therefore
 condition (\ref{assumpresvanishes}) holds. The previous lemma then gives
\begin{eqnarray}
\gr(0) H^n_{prim}  ( \cV(\Psi^e_{G}, \Psi_{G,e})) & = &  \gr(0) H^n_{prim}  ( \cV(\Psi^{ee'}_{G}, \Psi^{e'}_{G,e})) \nonumber \\
&  = &\gr(0) H^n_{prim}  ( \cV(\Psi^{e}_{G\backslash e'}, \Psi_{G\backslash e',e}))  \end{eqnarray}
where the second line follows by contraction-deletion. If $G\backslash e'$ is connected, it has one fewer edges and loops than $G$, so the induction goes through (if $G\backslash e'$ is not connected, its graph polynomial 
vanishes and $(\ref{PsiGupperlowerVanishes})$ holds trivially.)

Now we turn to $(\ref{PsiGVanishestoorder2})$, which is again proved by induction on the number of edges of $G$.
 Let us write  $\PP^{N-i}$ for $\PP(\alpha_i,\ldots, \alpha_N)$, when $i=1,2$. By contraction-deletion, we have
 $\Psi_G = \Psi_G^1 \alpha_1 + \Psi_{G,1}$. Let  $V=\cV(\Psi_G, \Psi_G^1)\subset \PP^{N-1}$ and let $U=\cV(\Psi_G)\backslash V$.  The localization sequence $(\ref{seq : loc1})$ gives
\begin{equation}\label{b35}
\rightarrow H^{n}_c(U)\rightarrow H^n(\cV(\Psi_G))\rightarrow
H^n(V)\rightarrow H^{n+1}_c(U)\rightarrow
\end{equation}
By the linearity of $\Psi_G$,  $U$ is isomorphic to $\PP^{N-2}\backslash\cV(\Psi_G^1)$.
Applying  the localization sequence $(\ref{seq : loc1})$ once again to the inclusion $\cV(\Psi_G^1) \subset \PP^{N-2}$, and taking primitive
cohomology, implies  that  for all $n$, 
\begin{equation} \label{p7indstep}
 H^n_c(U) \cong H^{n}_c(\PP^{N-2}\backslash\cV(\Psi_G^1)) \cong  H^{n-1}_{prim}(\cV(\Psi_G^1))\ . 
 \end{equation}
 By contraction-deletion, $\Psi_G^1= \Psi_{G \backslash 1}$, where $G\backslash 1$ is either disconnected, or   has one fewer edges and loops than $G$. Therefore by induction hypothesis
 we have  
 \begin{equation}\label{Uvanishestoorder2}
 \gr(i)  H^n_c(U) = 0 \hbox{ for } i =0,1\ .
 \end{equation}
Now consider the projection $\PP^{N-1}\backslash\,  p\rightarrow
\PP^{N-2}$ from the point  $p= (1: 0 : \ldots :0)$. It follows from the shape of $\Psi_G$ that 
$V$ is a cone over $\cV(\Psi_G^1,\Psi_{G,1})\subset \PP^{N-2}$, and therefore by lemma \ref{lemcone},
\begin{equation}\label{b35.5}
H^n_{prim}(V)\cong H^{n-2}_{prim}(\cV(\Psi_G^1,\Psi_{G,1}))(-1)
\end{equation}
for all $n$.  By equation $(\ref{PsiGupperlowerVanishes})$, it follows that 
$$\gr(i) H^n_{prim}(V) =0 \hbox { for } i =0,1\ .$$
Combining this with $(\ref{Uvanishestoorder2})$ and applying the grading functors to the sequence $(\ref{b35})$, we conclude that
$\gr(i) H^n(\cV(\Psi_G)) =0 $ for $i=0,1$ as required.
\end{proof}

\begin{lemma} \label{lem2valent} Suppose that $G$ is connected, satisfies $h_G \leq N_G-3$,  and   has a $2$-valent vertex. Then in addition,   for all $n$, 
$$\gr(2) H^n_{prim}(X_G) =0\ .$$
\end{lemma}
\begin{proof}
Let the edges incident to the $2$-valent vertex be $1,2$. Then 
$$\Psi_G = \Psi_{G\backslash 2 \q 1} (\alpha_1+\alpha_2) + \Psi_{G\q 1,2}\ ,$$
which follows from contraction-deletion. By changing variables, one sees that 
 $X_G$ is a cone over $X_{G\q 1}$ where 
$\Psi_{G\q 1} = \Psi_{G\backslash 2\q 1} \alpha_2 + \Psi_{G\q 1,2}.$ Hence by lemma \ref{lemcone}, 
$$H^n_{prim} (X_G ) = H^{n-2}_{prim} (X_{G\q 1})(-1)\ ,$$
and the conclusion follows immediately from equation $(\ref{PsiGVanishestoorder2})$.
\end{proof}

\subsection{Initial reductions}
We have shown that under some mild conditions 
$$\gr(i) H^n (X_G) =0 \hbox{ for } i = 0, 1\ .$$
The next goal is to compute the first non-trivial piece, $\gr(2) H^n (X_G)$ in terms of some hypersurfaces defined by some related polynomials.
For this, it is convenient to assume that $G$ has  a  three-valent vertex. Note that if   $G$ is connected and  $2 h_G \leq N+1$ (the case of interest)  then $G$ automatically has a vertex of degree at most three. The case of a two-valent vertex
is trivial and covered by lemma  \ref{lem2valent}. 

\begin{proposition}\label{p3}
Let $G$ be a connected graph  with a  3-valent vertex, satisfying $h_G \leq N_G-2$.  Denote the edges incident to this vertex by $1,2,3$. Then 
\begin{equation}
\gr(2) H^{n}_{prim}(X_G) =  \gr(0)
H^{n-4}_{prim}(\cV(\Psi^{13,23}_G, \Psi^{1,2}_{G,3})) 
\end{equation} 
for all $n$.
\end{proposition}
 \begin{proof}
Write $\PP^{N-1} =\PP^{N-1}(\alpha_1, \ldots, \alpha_N)$ and  $\PP^{N-4} =\PP^{N-4}(\alpha_4, \ldots, \alpha_N)$.
We stratify $X_G\subset \PP^{N-1}$ in a similar manner to  Proposition 23 in \cite{BrSch}.

Let $f_0,f_1,f_2,f_3, f_{123}$ be the polynomials defined by $(\ref{b9})$, and recall that  the graph polynomial $\Psi_G$ can be expressed in the 
 form $(\ref{Psistar})$, which we repeat here:
$$f_0(\alpha_1\alpha_2+\alpha_1\alpha_3+\alpha_2\alpha_3)+(f_1+f_2)\alpha_3+(f_2+f_3)\alpha_1+(f_1+f_3)\alpha_2+f_{123},
$$ 
where $f_0 f_{123} = f_1f_2+f_1f_3+f_2f_3$. 
The closed subscheme $X_G\cap \cV(f_0)\subset X_G$ gives rise to  the following exact  localization sequence $(\ref{seq : loc1})$:
\begin{equation}\label{b14}
\longrightarrow H^{n}_c(U_1)\longrightarrow H^{n}_{prim}(X_G)\longrightarrow
H^{n}_{prim}(X_G\cap\cV(f_0))\longrightarrow H^{n+1}_c(U_1)\longrightarrow
\end{equation}
where $U_1:=X_G\backslash (X_G\cap\cV(f_0))$. 
Consider the projection $\pi:\PP^{N-1}\backslash S\rightarrow\PP^{N-4}$, where   $S$ denotes the plane
   $\cV(\alpha_4,\ldots, \alpha_N)$, and let 
  $U'_1:=\pi(U_1)$. After  making the change of variables $\beta_i=f_0\alpha_i+f_i$ for $i=1,2,3$, we find that 
\begin{equation}
f_0\Psi_G=\beta_1\beta_2+\beta_1\beta_3+\beta_2\beta_3\ .
\end{equation}
The right-hand side  defines an affine quadric $Q$ in $\AAA^3$. Since the above change of coordinates is invertible outside $\cV(f_0)$,  we have
   $\pi : U_1 \overset{\sim}{\rightarrow} Q \times U_1'$, where 
   $U_1' \cong \PP^{N-4} \backslash \cV(f_0)$.
    One easily shows that 
$H^{k}_c(Q)$ is concentrated in degree four and $H^{4}_c(Q)\cong \QQ(-2)$. Therefore  by  the K\"{u}nneth formula,
  \begin{equation}\label{b16}
 H^{n}_c(U_1)\cong H^4_c(Q) \otimes H^{n-4}_c(U'_1) \cong   H^{n-4}_c(\PP^{N-4} \backslash \cV(f_0))(-2)\ , 
\end{equation}
for all $n$.  A localization sequence $(\ref{seq : loc1})$ applied to $\cV(f_0) \subset \PP^{N-4}$ gives
$$  H^{m-1}_{prim}( \cV(f_0) ) \cong H^{m}_c( \PP^{N-4} \backslash \cV(f_0)  )\ .$$
By equation $(\ref{b9})$, $\deg f_0 = h_G-2$ and hence  $\deg f_0 \leq N-4$ by assumption. It follows from the Chevalley-Warning theorem  \ref{t20} that 
 $\gr(0) H^{m-1}_{prim}( \cV(f_0) )=0$. By $(\ref{b16})$ we deduce that $\gr(i) H^{n}_c(U_1) =0$ for $i\leq 2$, and therefore
 by applying the grading functors to $(\ref{b14})$, we have 
\begin{equation}\label{b17}
 \gr(i) H^{n}_{prim}(X_G)\cong\gr(i) H^{n}_{prim}(X_G\cap\cV(f_0))
\end{equation}
for $i \leq 2$.  Now consider  $X_G\cap\cV(f_0)$. By $(\ref{Psistar})$,
\begin{equation}\label{Psiatf0}
\Psi_G|_{f_0=0}=(f_1+f_2)\alpha_3+(f_1+f_3)\alpha_2+(f_2+f_3)\alpha_1+f_{123}.
\end{equation}
Let $Y:=X_G\cap\cV(f_0)\cap\cV(f_1,f_2,f_3)$. One has the localisation sequence
\begin{equation}\label{b20}
\rightarrow H^{n-1}_{prim}(Y)\rightarrow H^{n}_c(U_2)\rightarrow
H^{n}_{prim}(X_G\cap\cV(f_0))\rightarrow H^{n}_{prim}(Y)\rightarrow,
\end{equation}
where $U_2:=X_G\cap\cV(f_0)\backslash Y$.  From equation $(\ref{Psiatf0})$ we  have  
$Y = Y \cap \cV(f_{123})$ and $Y \cap  \cV(f_{123}) \cong \AAA^3 \times \cV(f_0,f_1,f_2,f_3,f_{123})$, and hence
\begin{equation}
H^{n}_{prim}(Y)\cong H^{n}_{prim}(Y\cap\cV(f_{123}))\cong
H^{n-6}_{prim}(\cV(f_0,f_1,f_2,f_3,f_{123}))(-3)
\end{equation}
for all $n$, where  $\cV(f_0,f_1,f_2,f_3,f_{123})\subset \PP^{N-4}$.
Applying  $\gr(i)$ to (\ref{b20}) gives
\begin{equation}\label{b23}
\gr(i) H^{n}_{prim}(X_G\cap\cV(f_0))\cong \gr(i) H^{n}_c(U_2)
\end{equation}
for $i\leq 2$.  Equation $(\ref{Psiatf0})$ defines a family of non-degenerate hyperplanes over  
 $U_2'=\pi(U_2)=\cV(f_0)\backslash\cV(f_0,f_1,f_2,f_3)$, so  
 $U_2$ is an $\AAA^2$-bundle over $U_2'$. Since $\AAA^2$ has trivial cohomology (or by  applying the localization sequence $(\ref{seq : loc1})$ successively with respect to 
 $f_i+f_j=0$), it is easy to see that 
 \begin{equation}\label{b24}
H^{n}_c(U_2)\cong H^{n-4}_c(U_2')(-2).
\end{equation}
To compute the cohomology of $U_2'$, we use  the exact sequence
\begin{multline}
\rightarrow H^{n-5}_{prim}(\cV(f_0)) \rightarrow H^{n-5}_{prim}(\cV(f_0,f_1,f_2,f_3))\rightarrow\\ \rightarrow H^{n-4}_c(U'_2)\rightarrow H^{n-4}_{prim}(\cV(f_0)) \rightarrow
\end{multline}
We have already shown that $\gr(0) H^m(\cV(f_0))$ vanishes for all $m$, so the previous sequence implies that
\begin{equation} \label{b25}
  \gr(0)  H^{n-4}_c(U'_2) \cong  \gr(0)  H^{n-5}_{prim}(\cV(f_0,f_1,f_2,f_3)) \ .
  \end{equation}
The final step is to eliminate some of the   $f_i$'s.
For this,  observe  that by (\ref{b11}), 
$\cV(f_0,f_3) =\cV(f_0,f_1f_2,f_3)$. Therefore  
a Mayer-Vietoris sequence (\ref{seq :  MV}) gives
\begin{multline}\label{b32}
\rightarrow H^{n-5}_{prim}(\cV(f_0,f_1,f_3))\oplus H^{n-5}_{prim}(\cV(f_0,f_2,f_3)) \\
\rightarrow H^{n-5}_{prim}(\cV(f_0,f_1,f_2,f_3))\rightarrow H^{n-4}_{prim}(\cV(f_0,f_3))\\
\rightarrow H^{n-4}_{prim}(\cV(f_0,f_1,f_3))\oplus
H^{n-4}_{prim}(\cV(f_0,f_2,f_3)) \rightarrow.
\end{multline}
Now consider $\cV(f_0,f_1,f_3)$. By (\ref{b11}),
$$\cV(f_0,f_1+f_3)\cong\cV(f_0,f_1+f_3,f_1f_3)\cong
\cV(f_0,f_1,f_3)\ . $$
By (\ref{b11}), or by contraction-deletion, 
$\cV(f_0,f_1+f_3)\cong \cV(\Psi_{G'}^1,\Psi_{G',1})$ where $G'$ is the graph 
$G\backslash \{2\}\q3$. By 
 $(\ref{PsiGupperlowerVanishes})$,   $\gr(0) H_{prim}^{n-4} (\cV(\Psi_{G'}^1,\Psi_{G',1})) =0 $. We deduce that
 $$\gr(0) H^{n-4}_{prim}(\cV(f_0,f_1,f_3))=0\ ,$$
 and similarly for  $H^{n-4}_{prim}(\cV(f_0,f_2,f_3))$. Sequence $(\ref{b32})$ gives
 \begin{equation}\label{f0f3}
  \gr(0) H^{n-5}_{prim}(\cV(f_0,f_1,f_2,f_3))\cong \gr(0) H^{n-4}_{prim}(\cV(f_0,f_3))\ .
  \end{equation} 
Putting the isomorphisms    (\ref{b17}), (\ref{b23}), (\ref{b24}),  (\ref{b25}), (\ref{f0f3}) and  together  gives
$$ \gr(2) H^{n}_{prim}( X_G ) \cong \gr(0) H^{n-4}_{prim}(\cV(f_0,f_3))\ .$$
By (\ref{b11}),  and the remarks following it, $f_3 = \Psi^{1,2}_{G,3}$, and $f_0= \Psi^{12}_{G,3} =\Psi^{13,23}_{G}$. 
\end{proof}
 
\subsection{Denominator reduction in cohomology}
We finally restrict to the physically interesting case:  $G$ is connected and  overall log-divergent, i.e.,  
$$
N_G= 2h_G\ , \hbox{ and } N_G\geq 5\ .$$
Let $D_0= \Psi_G, D_1, \ldots, D_k$ denote the first $k\geq 5$ polynomials   in the denominator reduction with respect to some ordering on the edges of $G$.  

\begin{theorem} \label{thmcohomDR} Let $G$ be as above. Then for all $n$,
\begin{eqnarray}
\gr(i) H^{n}_{prim}(X_G)  &  = & 0 \quad   \hbox{Êfor } i =0,1 \ ,\nonumber \\ 
\gr(2) H^{n}_{prim}(X_G)  &  \cong  & \gr(0) H^{n-k}_{prim} (D_k) \ ,
\end{eqnarray}
for all $k\geq 3$ for which $D_k$ is defined.
\end{theorem}
\begin{proof} The vanishing of 
$\gr(i) H^{n}_{prim}(X_G) $ for $i=0,1$ follows from equation  $(\ref{PsiGVanishestoorder2})$.
The conditions on $G$ imply that it has a vertex of valency  $3$ or less.
Suppose first of all  that it has a three-valent vertex with incident edges $1,2,3$. Proposition \ref{p3} implies that 
$$  \gr(2) H^{n}_{prim}(X_G)  \cong \gr(0) H^{n-4}_{prim}( \cV( \Psi^{13,23}_G , \Psi^{1,2}_{G,3}))\ .$$
Two applications of the generic denominator reduction step (propositions  \ref{propDenomRed} and \ref{p5}) with respect to a further two edge variables $4$ and $5$ gives
$$\gr(0) H^{n-4}_{prim}( \cV( \Psi^{13,23}_G , \Psi^{1,2}_{G,3})) \cong \gr(0) H^{n-5}_{prim}( \cV( {}^5\Psi(1,2,3,4,5)_G ))\ ,$$
where ${}^5\Psi(1,2,3,4,5)_G$ is the five-invariant of $G$ with respect to these edges.  Since the  previous equation holds for any  five edges of $G$, and since the  vanishing locus of the five-invariant does not depend on the order  of the variables,  the  equation 
\begin{equation} \label{gr2asfiveinv}
 \gr(2) H^{n}_{prim}(X_G)  \cong     \gr(0) H^{n-5}_{prim}( \cV( {}^5\Psi(1,2,3,4,5)_G )) 
 \end{equation}
in fact  holds for any set of five edges of $G$ (not necessarily containing a 3-valent vertex). We may therefore assume that the edges $1,\ldots, 5$ are the first five edges in the denominator reduction. Thus  we have
$$ \gr(2) H^{n}_{prim}(X_G)  \cong     \gr(0) H^{n-5}_{prim}( \cV( D_5)) \ .$$
It follows by induction  by $(\ref{DenomGenericStepVersion2})$  that
$$    \gr(0) H^{n-m}_{prim}( \cV( D_m))   \cong    \gr(0) H^{n-m-1}_{prim}( \cV( D_{m+1}))\ ,$$ 
for all $ m \geq 5$ such that $D_{m+1}$ is defined (and, clearly,  for $m=3,4$ as well).   Condition (\ref{p5assumption})
holds since $\cV(D_k) \subset \PP^{N-k-1}$ is of degree $2 h_G - k=N-k$.

Now consider the case when $G$ has a two-valent vertex. By lemma \ref{lem2valent}, we know that  $\gr(2)  H^{n}_{prim}(X_G) =0$, and
we know by \cite{BY}, lemma 92,  that the five-invariant vanishes in this case also.  So $(\ref{gr2asfiveinv})$ holds and the argument is as before.
The remaining cases, when $G$ has a one-valent vertex or a three-valent vertex with a self-loop, are even more trivial and left to the reader.
\end{proof}

\begin{corollary} \label{corwd} Suppose that  $G$ is  connected,  denominator reducible,  and satisfies $2h_G\leq  N_G\geq 5$. 
If $G$ has weight-drop, or $2h_G <N_G$,  then 
$$ \gr(2)  H^{n}_{prim}(X_G)=0  \quad \hbox{ for all } n \ .$$
If $G$ does not have weight-drop, then 
\begin{equation} 
\gr(2) H^{n}_{prim}(X_G)\cong\left\{
\begin{aligned}
\QQ(-2) \ , \;\; \hbox{ if } n=N_G-2 \\
 0 \ , \;\; \quad \quad  \hbox{otherwise}.
\end{aligned}  \right.
\end{equation}
\end{corollary}
\begin{proof}
The weight-drop case follows immediately from the previous theorem. The case when $2h_G<N_G$ follows from the previous 
theorem combined with the Chevalley-Warning theorem \ref{t20}.
In the other case,  the final stage in the denominator reduction is a polynomial
$D_{N-2}$ of bidegree $(1,1)$  in two variables.   Therefore  $\cV(D_{N-2}) \subset \PP^1$  is isomorphic to two distinct points, and  satisfies 
$H_{prim}^n( \cV(D_{N-2}) )= \QQ(0)$ if $n=0$, and vanishes otherwise. The result then follows from the previous theorem.
\end{proof}

\newpage

\section{Reduction of differential forms} \label{sectRedForms}

\subsection{Smoothness results}
We prove some preliminary results on the smoothness of certain
complements of graph hypersurfaces. Let  $e$ be an edge of $G$. The
following proposition was proved in \cite{Pat}, \cite{BSY}.

\begin{proposition} \label{propsmooth} The hypersurface complement  $X_{G\backslash e} \backslash (X_{G\backslash e} \cap X_{G\q e})$ is smooth.
\end{proposition}
\begin{proof} We repeat the proof from \cite{BSY}.  If $G$ has no loops then the result is trivial.  Number the edges of $G$ so that  $e$ is denoted $1$, and $1, 2, \ldots, k$ forms a cycle in $G$. The first observation (proposition     24   in \cite{BSY}) is that 
\begin{equation}\label{PsiGlower1} \Psi_{G,1}= \sum_{j=2}^k  \lambda_j x_j \Psi_G^{1,j}, \hbox{ for some } \lambda_j =\pm 1
\end{equation}
Let $I$ be the ideal in $\mathbb{Q}[\alpha_i]$ spanned by $
\Psi_G^{1}, {\partial \Psi_G^1 \over \partial \alpha_2} ,\ldots,
{\partial  \Psi_G^1 \over \partial \alpha_k}$. By linearity this is
the ideal spanned by $ \Psi_G^{1}, \Psi_G^{12}, \ldots, \Psi_G^{1k} $.
 It follows from the Dodgson identity that
$$(\Psi_G^{1,j})^2 = \Psi_{G,j}^1\Psi_{G,1}^{j}- \Psi_{G,1j} \Psi_G^{1j} =\Psi_G^1\Psi_{G,1}^{j}- \Psi_{G,1} \Psi_G^{1j}\in I \ , $$
and so $\Psi_G^{1,j}\in \sqrt{I}$ for all $j\in E(G)$. By
$(\ref{PsiGlower1})$, this implies that  $\Psi_{G,1} \in \sqrt{I}.$ This
implies a fortiori that $\Psi_{G,1}$ vanishes on the singular locus of
$\cV(\Psi_G^1)$. The statement that $X_{G\backslash 1} \backslash (X_{G\backslash 1} \cap X_{G\q 1})$ is smooth  follows  from
$\Psi_{G}^1=\Psi_{G \backslash 1}$ and $\Psi_{G,1} = \Psi_{G\q 1}$,  which is  simply the  contraction-deletion relation.
\end{proof}

This result probably generalizes to the zero loci of all Dodgson
polynomials (smoothness  on the set of complex points  follows by
Patterson's theorem, which holds generally for configuration
polynomials). We only need the following special case.  Let $G$ be a
graph with a 3-valent vertex, which meets edges numbered $1,2,3$.
\begin{corollary} \label{corsmooth} The  open scheme  $\cV(\Psi_G^{13,23}) \backslash
( \cV(\Psi_G^{13,23})\cap \cV(\Psi_{G,3}^{1,2}))   $
is smooth. \end{corollary}

\begin{proof} 
We again use the structure of the graph polynomial of a graph with a
3-valent vertex, see $(\ref{Psistar})$.  
Since $\Psi_G^{13,23} =\Psi_{G,1}^{23} =f_0$,
$$\cV(\Psi_G^{13,23}) \cap \cV(\Psi^{1,2}_{G,3}) = \cV(f_0,f_3) \ .$$
On the other hand,
$$\cV(\Psi^{23}_{G,1}) \cap \cV(\Psi^{2}_{G,13}) = \cV(f_0, f_1+f_3) =\cV(f_0,f_1,f_3) \ , $$
where the second equality follows from $(\ref{b11})$. It follows
that
$$\cV(\Psi_G^{13,23}) \backslash ( \cV(\Psi_G^{13,23})\cap \cV(\Psi^{1,2}_{G,3}))\quad  \subset \quad \cV(\Psi^{23}_{G,1}) \backslash (\cV(\Psi^{23}_{G,1}) \cap \cV(\Psi^{2}_{G,13}))\ .$$
By contraction-deletion, the right-hand side is precisely
$X_{H\backslash 3} \backslash (X_{H\backslash 3} \cap X_{H\q 3})$,
where $H =G\backslash 2\q 1$, which is smooth by proposition
\ref{propsmooth}. 
\end{proof}

\begin{corollary}  \label{corsmooth2} For any graph $G$, $\cV(\Psi_G^{12}) \backslash (\cV(\Psi_G^{12}) \cap \cV(\Psi_G^{1,2}))$ is smooth.
\end{corollary}

\begin{proof}
The Dodgson identity implies that $$\cV(\Psi_G^{12}, \Psi_G^{1,2}) = \cV(\Psi_G^{12},  \Psi^{1}_{G,2} \Psi^2_{G,1})\ .$$
Therefore, by the previous proposition, 
$$\cV(\Psi_G^{12}) \backslash \cV(\Psi_G^{12}, \Psi_G^{1,2}) =  \cV(\Psi^{12}_G) \backslash  \cV(\Psi_G^{12},  \Psi^{1}_{G,2} \Psi^2_{G,1}) \subset \cV(\Psi^{12}_G) \backslash  \cV(\Psi_G^{12},  \Psi^{1}_{G,2} )$$
is smooth,  since it is $ \cV(\Psi_{H \backslash 2}) \backslash  \cV(\Psi_{H \backslash 2},  \Psi_{H\q 2} )$, where $H = G \backslash 1$. 
\end{proof}

\subsection{The generic reduction step for differential forms}
Let $f, g$ be two homogeneous polynomials of the form 
\begin{equation} \label{fandg}
f=f^1\alpha_1+f_1\ ,\;\; g=g^1\alpha_1+g_1
\end{equation}
where $f^1,f_1,g^1,g_1 \in \QQ[\alpha_2,\ldots, \alpha_N]$, whose resultant $[f,g]_{\alpha_1}$  is nonzero.   Let
$$X_f= \cV(f)\ , \  X_g= \cV(g)\ , \  X_{f^1}= \cV(f^1)\ , \ X_{g^1}= \cV(g^1)\ , \   X_{[f,g]}= \cV([f,g]_{\alpha_1})$$
Writing $\PP^{N-i} = \PP(\alpha_i : \ldots : \alpha_N)$ for $i=1,2$,  let  $\pi: \PP^{N-1} \rightarrow \PP^{N-2}$ denote the projection from the point $(1: 0: \ldots :0)$, and let 
$\widehat{X}_{f^1},\widehat{X}_{g^1}\subset\PP^{N-1}$ be the cones
 over $X_f$ and $X_g$. We have the following picture:
 \begin{eqnarray}
  \widehat{X}_{f^1},\widehat{X}_{g^1} , X_f, X_g  \subset  & \PP^{N-1}&  \nonumber \\
  & \downarrow_{\pi} & \nonumber \\
   X_{f^1}, X_{g^1}, X_{[f,g]}   \subset  & \PP^{N-2}&  \nonumber 
   \end{eqnarray}
Note that  $X_{[f,g]}= \pi( X_f \cap X_g)$ is the discriminant. Write
$$\Omega_N = \sum_{i=1}^N (-1)^i \alpha_i d\alpha_1 \wedge \ldots  \wedge\widehat{d \alpha_i} \wedge \ldots \wedge d \alpha_N\ .$$
\begin{proposition} \label{propgeneric}
If the  cohomology class 
 $$ \Big[
 {\Omega_{N-1}  \over fg} \Big] \in  \gr_F^{N-1} H^{N-1}_{dR}( \PP^{N-1} \backslash (X_f \cup X_g))  \ ,$$ 
vanishes, then so does
$$    \Big[ {\Omega_{N-2}  \over f^1g_1-f_1g^1}
\Big]      \in \gr_F^{N-2}   H^{N-2}_{dR}(\PP^{N-2}\backslash X_{[f,g]}) \ .$$
    \end{proposition}

\begin{proof}
First of all, let us assume that $f^1g^1$ is non-zero. Let $i_{f^1}$ be the
inclusion 
\begin{equation}
  i_{f^1}: \PP^{N-1}\backslash (X_f\cup X_g\cup \widehat{X}_{f^1}) \hookrightarrow \PP^{N-1}\backslash (X_f\cup
  X_g).
\end{equation}
Since $f^1=\frac{\partial f}{\partial \alpha_1}$,
the singular locus of $X_f$ is contained in $X_f\cap
\widehat{X}_{f^1}$ and hence $X_f\backslash X_f\cap
\widehat{X}_{f^1}$ is smooth. Therefore we have the residue map
\begin{equation} \nonumber
\Res_f : \Omega^{n}(\PP^{N-1}\backslash(X_f\cup X_g\cup
\widehat{X}_{f^1}))\longrightarrow \Omega^{n-1}(X_f\backslash((X_g\cup
\widehat{X}_{f^1})\cap X_f)) 
\end{equation}
Let us write
\begin{equation}
U_f=\PP^{N-2}\backslash(X_{[f,g]} \cup X_{f^1}) \quad\text{and}\;\;
U_g=\PP^{N-2}\backslash(X_{[f,g]} \cup X_{g^1})
\end{equation}
From the shape $(\ref{fandg})$ of $f$, the map $\pi$ induces  an isomorphism
\begin{equation}
X_f\backslash((X_g\cup \widehat{X}_{f^1})\cap X_f)  \overset{\sim}{\longrightarrow} U_f
\end{equation}
Composing  $\Res_f\circ i^*_{f^1}$ with the induced map on  forms gives \begin{equation}\label{b128}
R_f: \Omega^{n}(\PP^{N-1}\backslash(X_f\cup X_g))\longrightarrow
\Omega^{n-1}(U_f)  
\end{equation}
In an identical manner, we have a map
\begin{equation}
R_g:\Omega^{n}(\PP^{N-1}\backslash(X_f\cup X_g))\longrightarrow
\Omega^{n-1}(U_g) 
\end{equation}
On the level of cohomology, we have a map
$$[R_f]- [R_g]: H_{dR}^{N-1} ( \PP^{N-1}  \backslash (X_f \cup X_g)) \To \big( H_{dR}^{N-2}(U_f) \oplus   H_{dR}^{N-2}(U_g) \big)(-1)$$
 On the other hand, Mayer-Vietoris gives
 $$ \ldots \To H_{dR}^{N-2} ( U_f \cup U_g) \overset{r}{\To} H_{dR}^{N-2}(U_f) \oplus H_{dR}^{N-2}(U_g) \To H_{dR}^{N-2}(U_f \cap U_g) $$
Note that $U_f \cup U_g = \PP^{N-2} \backslash X_{[f,g]}$. We now show that 
\begin{equation} \label{formidentity1} r \Big(  \Big[ {\Omega_{N-2}  \over f^1g_1-f_1g^1}
\Big]    (-1)  \Big) = \big( [R_f] - [R_g] \big) \Big( \Big[
 {\Omega_{N-1}  \over fg} \Big] \Big)\ .\end{equation}
 This identity also holds on  the level of differential forms (rather than cohomology classes).
In order to compute the image of $\frac{\Omega_{N-1}}{fg}$ under the
map $R_f$, it is enough to work on an open subset of the form $\alpha_N\neq 0$.
We have
\begin{equation}
\omega = \frac{\Omega_{N-1}}{fg}\biggl|_{\alpha_N\neq 0} =
\frac{d\alpha_1\wedge\ldots\wedge
d\alpha_{N-1}}{f(g^1\alpha_1+g_1)}.
\end{equation}
The equation $f=f^1\alpha_1+f_1$ gives
\begin{equation}
df\wedge d\alpha_2\wedge\ldots\wedge d\alpha_{N-1}=
f^1d\alpha_1\wedge d\alpha_2\wedge\ldots\wedge d\alpha_{N-1}
\end{equation}
and therefore
\begin{equation}
i^*_{f^1}(\omega)=\frac{df}{f}\wedge\frac{d\alpha_2\wedge\ldots\wedge
d\alpha_{N-1}}{f^1(g^1\alpha_1+g_1)},
\end{equation}
whose residue along $X_f$ is therefore
\begin{equation}
\frac{d\alpha_2\wedge\ldots\wedge
d\alpha_{N-1}}{f^1(g^1\alpha_1+g_1)}\bigg|_{f^1\alpha_1+f_1=0} =
\frac{d\alpha_2\wedge\ldots\wedge d\alpha_{N-1}}{f^1g_1-g^1f_1} 
\end{equation}
By  a similar calculation for $R_g$, we have 
$$R_f \Big ( {\Omega_{N-1}  \over fg}\Big) =    {\Omega_{N-2} \over f^1g_1-g^1f_1}\Big|_{U_f}\quad \hbox{ and } \quad R_g \Big ( {\Omega_{N-1}  \over fg}\Big) =    {\Omega_{N-2} \over g^1f_1-f^1g_1}\Big|_{U_g}$$
which proves $(\ref{formidentity1})$ as required. To conclude the proof, apply the graded functors $\gr_F^{N-1}$ to both sides of $(\ref{formidentity1})$. It suffices to show that the map
$$ \gr_F^{N-2}( r) :  \gr_F^{N-2} H_{dR}^{N-2} ( U_f \cup U_g) {\To}\gr_F^{N-2} H_{dR}^{N-2}(U_f) \oplus H_{dR}^{N-2}(U_g)  $$
is injective. But this follows from the fact that its kernel vanishes:
$$ \gr_F^{N-2} H_{dR}^{N-3}(U_f \cap U_g) =0 \ .$$
This is an immediate consequence of the bounds on the Hodge numbers  $(\ref{Xhodgenos})$.
The case when $f^1$ or $g^1$ vanishes is similar and left to the reader.
\end{proof}
Unfortunately, the previous proposition does not cover the case of the second reduction step, which involves the slice one lower in the Hodge filtration.
\begin{proposition} 
\label{propsecondDRstep} Let $G$  be a connected  graph with $N$ edges and a three-valent vertex with incident edges $e_1,e_2,e_3$.  Then if the class
$$ \Big[{\Omega_{N-2} \over \Psi_G^{e_1} \Psi_{G,e_1}} \Big] \in \gr^{N-3}_F H_{dR}^{N-2}(\PP^{N-2} \backslash (\cV({\Psi^{e_1}_G}) \cup \cV({\Psi_{G,e_1}}))) $$
vanishes, then so does
$$\Big[{\Omega_{N-3} \over  (\Psi_G^{e_1,e_2})^2 }  \Big] \in \gr^{N-4}_F H_{dR}^{N-3} (\PP^{N-2} \backslash \cV(\Psi_G^{e_1,e_2})) \ .$$
\end{proposition}
\begin{proof}
We use the same  notations as in  the proof of the previous proposition. Here $f= \Psi_G^{e_1}$,  $g= \Psi_{G,e_1}$, $f^1= \Psi_G^{e_1 e_2}$, and 
$[f,g] = (\Psi_G^{e_1,e_2})^2$ by the Dodgson identity. This time we only need to consider a single  residue map $R_f$
$$ R_f: H_{dR}^{N-2}(\PP^{N-2}\backslash ( X_f \cup X_g)) \To H_{dR}^{N-3} (U_f) (-1)\ .$$
It suffices to show that the inclusion of open sets induces  an injection
\begin{equation} \label{grN-4map} \gr^{N-4}_F  H_{dR}^{N-3} (U_f \cup U_g) \To  \gr^{N-4}_F H_{dR}^{N-3} (U_f)     \ .
\end{equation} 
For this,  consider the Gysin sequence associated to the inclusion
$$X_{f^1} \backslash (X_{f^1} \cap  X_{[f,g]}) \subset \PP^{N-3} \backslash X_{[f,g]} = U_f \cup U_g\ .$$
Note that $X_{f^1} \backslash (X_{f^1} \cap X_{[f,g]}) $ is smooth by corollary \ref{corsmooth2}.   
The kernel of $(\ref{grN-4map})$ is
\begin{equation} \label{ker73}
 \gr^{N-4}_F  \big( H_{dR}^{N-5} (X_{f^1} \backslash (X_{f^1} \cap  X_{[f,g]}) )(-1)\big)\ .
 \end{equation} 
From the structure of a 3-valent vertex $(\ref{Psistar})$, $f^1$ is equal to $f_{0}$ and $[f,g]$ is equal to $f_{0} \alpha_{e_3} + f_{e_3}$.  It follows that
$ X_{f^1} \backslash (X_{f^1} \cap  X_{[f,g]})$ is an $\AAA^1$-fibration over $\cV( f_{0}) \backslash \cV(f_{0}, f_{e_3})$. The  expression $(\ref{ker73})$ is therefore
$$  \gr^{N-5}_F   H_{dR}^{N-6} (\cV( f_{0}) \backslash \cV(f_{0}, f_{e_3}) ) = 0 \ .$$
\end{proof}

\subsection{The degenerate (weight drop) case}
Let $\Psi\in \QQ[\alpha_1,\ldots, \alpha_{N}]$ be a homogeneous
polynomial  of the form
$$\Psi= \Psi^e \, \alpha_e + \Psi_e\ ,$$
where $\Psi^e, \Psi_e \in   \QQ[\alpha_1,\ldots, \widehat{\alpha}_e,
\ldots,  \alpha_{N}]$ do not depend on the variable $\alpha_e$.  Let us write
$\PP^{N-1} = \PP^{N-1}(\alpha_1 : \ldots : \alpha_N)$ and $\PP^{N-2} = \cV(\alpha_e) \overset{i}{\hookrightarrow}  \PP^{N-1}$. Let
$$X= \cV(\Psi) \subset \PP^{N-1} \ ,  \ X^e= \cV(\Psi^e) \ , \  ÊX_e= \cV(\Psi_e) 
\subset \PP^{N-2}$$ 
and let   $\Xh^e$ and $\Xh^e$ in $\PP^{N-1}$ denote  the cones
over $X^e$ and $X_e$. 
\begin{proposition} \label{propdegen}
Suppose that  $X^e \backslash (X^e \cap X_e)$ is smooth.
Then there is a map
\begin{equation} \label{degendRmap}
H_{dR}^{N-2} (\PP^{N-2} \backslash (X^e \cup X_e)) \To
 H_{dR}^{N-1}(\PP^{N-1} \backslash X)
 \end{equation}
which maps the cohomology class 
\begin{equation}\label{b85}
  \Big[ {\Omega_{N-2} \over
\Psi^e \Psi_e} \Big ] \quad \hbox{ to } \quad    \Big[   {\Omega_{N-1} \over  \Psi^2} \Big] \ .
\end{equation}
Furthermore, suppose that for  all $n$,  
\begin{eqnarray} \label{propdegenvanishcond} \gr(i) H_{prim}^n(X^e) & = &0 \quad  \hbox{ for all } \quad  0\leq i \leq k+1 \ ,    \\
  \gr(i) H_{prim}^n(X_e)  & = &0 \quad  \hbox{ for all }  \quad 0\leq i \leq k  \ . \nonumber 
\end{eqnarray}Then $(\ref{degendRmap})$ induces an isomorphism:
$$
   \gr^{N-2-k}_F H_{dR}^{N-2} (\PP^{N-2} \backslash (X^e \cup X_e)) \cong
 \gr^{N-2-k}_F H_{dR}^{N-1}(\PP^{N-1} \backslash X)\ .$$ 
\end{proposition}

\begin{proof} By the smoothness assumption, we have one Gysin sequence coming 
from the inclusion of $X^e  \backslash (X^e \cap X_e) $
into $\PP^{N-2}\backslash X_e$:
\begin{multline}  \label{G1}
\cdots \To H_{dR}^{N-2}(\PP^{N-2}\backslash X_e)  \To H_{dR}^{N-2}(\PP^{N-2} \backslash (X^e \cup X_e)) \\
 \overset{\mathrm{Res}}{\To} H_{dR}^{N-3}(X^e \backslash (X^e \cap X_e))(-1) \To  H_{dR}^{N-1}(\PP^{N-2}\backslash X_e)  \To\cdots \ ,
\end{multline}
and another  from the inclusion of $\Xh^e\backslash ( \Xh^e\cap
X)= \Xh^e\backslash ( \Xh^e\cap
\Xh_e) $ into $\PP^{N-1} \backslash X$:
\begin{multline} \label{G2}
\cdots \To  H_{dR}^{N-2}(\PP^{N-1} \backslash  (X \cup  \Xh^e))  \To
H_{dR}^{N-3}(\Xh^e\backslash ( \Xh^e\cap \Xh_e) )(-1)\\
\overset{\gamma}{\To} H_{dR}^{N-1}(\PP^{N-1} \backslash X)
 \To  H_{dR}^{N-1}(\PP^{N-1} \backslash  (X \cup  \Xh^e))  \to  \cdots
\end{multline}
Since $\Xh^e\backslash ( \Xh^e\cap \Xh_e) $ is
an $\mathbb{A}^1$-fibration over $X^e\backslash ( X^e\cap X_e) $, we
have
$$  i^*: H_{dR}^{N-3}(\Xh^e\backslash ( \Xh^e\cap \Xh_e) ) \cong H_{dR}^{N-3}(X^e \backslash (X^e \cap X_e))$$
and the desired map  $(\ref{degendRmap})$ is
$$\gamma \circ (i^*)^{-1} \circ \mathrm{Res} : H_{dR}^{N-2}(\PP^{N-2} \backslash (X^e \cup X_e)) \To H_{dR}^{N-1}(\PP^{N-1} \backslash X)\ .$$
We now  wish to compute the image of $(\ref{b85})$ under this map.
  We use the maps
$i$ and $j$ indicated in the  following  diagram (where
all maps are inclusions)
$$
\begin{array}{ccc}
  \PP^{N-1} \backslash (X \cup \Xh^e) & \overset{j}{\To}    & \PP^{N-1} \backslash X
 \\
 \uparrow_i  &   &    \uparrow \\
 \PP^{N-2} \backslash (X^e \cup X_e)   & \To   &  \PP^{N-2} \backslash X_e 
\end{array}
$$
It suffices to  calculate with the restriction of our differential
forms to some  open affine subset   $\alpha_{N}\neq 0$ (where  $e\neq N$). We work with the forms
$$ \omega_1 =   { \Omega_{N-1} \over \Psi^2} \Big|_{\alpha_{N} =1}= {d\alpha_1\wedge \ldots  \wedge d\alpha_{N-1} \over  (\Psi^e\alpha_e + \Psi_e)^2} \ ,   $$
$$ \omega_2 =   { \Omega_{N-2} \over \Psi^e \Psi_e} \Big|_{\alpha_{N}= 1}= {d\alpha_1 \wedge\ldots  \wedge \widehat{d \alpha_e} \wedge \ldots  \wedge d\alpha_{N-1} \over \Psi^e  \Psi_e} \ ,  $$
$$ \omega_3=   {d\alpha_1 \wedge\ldots  \wedge \widehat{d \alpha_e} \wedge \ldots  \wedge d\alpha_{N-1} \over \Psi^e ( \Psi^e \alpha_e + \Psi_e)} \ ,  $$
and we can evidently ignore all signs.
By a trivial computation,  $j^* \omega_1=  d  \omega_3 $, and
therefore by the following exact commutative diagram:
$$\begin{array}{ccccc}
\Omega^{N-2}(\PP^{N-1} \backslash X) & \overset{j^*}{\To} & \Omega^{N-2}(\PP^{N-1} \backslash ( X \cup \Xh^e))&  \overset{\Res}{\To} & \Omega^{N-3}(\Xh^e \backslash ( \Xh^e \cap \Xh_e))   \nonumber \\
 \downarrow_d  & & \downarrow_d & & \downarrow_d \nonumber \\
 \Omega^{N-1}(\PP^{N-1} \backslash X) & \overset{j^*}{\To} & \Omega^{N-1}(\PP^{N-1} \backslash ( X \cup \Xh^e))&  \overset{\Res}{\To} & \Omega^{N-2}(\Xh^e \backslash ( \Xh^e \cap \Xh_e))
\end{array}$$
 and the definition of $\gamma$ as a boundary map, we have
$$[\omega_1] = \gamma ( [\mathrm{Res}_{\Xh^e\backslash (\Xh^e\cap \Xh_e)} \omega_3] ) \ . $$
Evidently, $\omega_2 =i^* \omega_3$, and so we have the 
identity
$$[\omega_1] =   \gamma \circ  (i^*)^{-1} \circ   \mathrm{Res}_{X^e \backslash (X^e\cap X_e)} [\omega_2]\ .$$
This proves $(\ref{b85}).$

For the second part, the localisation sequence $(\ref{seq : loc1})$ for $X_e\subset \PP^{N-2}$ gives
$$H^n_{prim}(X_e) \cong H^{n+1}_c (\PP^{N-2} \backslash X_e) $$
for all $n$.  Since $\PP^{N-2} \backslash X_e$ is smooth, duality reads 
$$H^{2N-4-m} (\PP^{N-2} \backslash X_e)^{\vee}  \cong H^{m}_c (\PP^{N-2} \backslash X_e)(N-2) \  . $$
The assumption $(\ref{propdegenvanishcond})$ for $X_e$ therefore implies that
$$\gr_F^{N-2-i}  H_{dR}^{n} (\PP^{N-2} \backslash X_e)=0 \quad  \hbox{ for all } 0\leq i \leq k\  . $$
By the exact sequence $(\ref{G1})$, it follows that $\Res$ induces an isomorphism on the corresponding graded pieces. 

It follows from the equation  $\Psi= \Psi^e\alpha_e+\Psi_e$ 
that  $\PP^{N-1} \backslash  (X \cup  \Xh^e)$ is a $\GG_m$-fibration over $\PP^{N-2}\backslash X^e$. Therefore,  for all $n$:
\begin{equation} \label{Gmbundledegencase}
  H_{dR}^{n}(\PP^{N-1} \backslash  (X \cup  \Xh^e))\cong H_{dR}^{n}(\PP^{N-2}\backslash
  X^e) \oplus H_{dR}^{n-1}(\PP^{N-2}\backslash
  X^e)(-1)  \ .
\end{equation}
By the assumption  $(\ref{propdegenvanishcond})$ for $X^e$, and by a similar argument to the  above, we have
$$\gr_F^{N-2-i}  H_{dR}^{n} (\PP^{N-2} \backslash X^e)=0 \quad  \hbox{for all}  \quad 0\leq i \leq k+1\  . $$
Combining this with $(\ref{Gmbundledegencase})$, we get
$$\gr_F^{N-2-i}  H_{dR}^{n} (\PP^{N-1} \backslash  (X \cup  \Xh^e))=0 \quad  \hbox{for all}  \quad 0\leq i \leq k\  , $$
and it follows from the exact sequence $(\ref{G2})$ that $\gamma$ induces an isomorphism on the corresponding graded pieces also.
\end{proof}
\subsection{Denominator reduction for framings}
We apply the previous results to graph hypersurfaces.

\begin{proposition} Let $G$ be a connected graph satisfying $2h_G \leq N_G \geq 6$.  Suppose that $G$ has a three-valent vertex with incident edges $1,2,3$.
If
$$\Big[ {\Omega_{N-1} \over \Psi_G^2 } \Big] \in \gr_F^{N- 3} H_{dR}^{N-1}(\PP^{N-1} \backslash X_G) $$
vanishes then so does
$$ \Big[{ \Omega_{N-4} \over \Psi_G^{13,23}\Psi^{1,2}_{G,3}} \Big] \in \gr_F^{N- 4}  H_{dR}^{N-4}(\PP^{N-4} \backslash \cV(\Psi_G^{13,23}, \Psi^{1,2}_{G,3})) \ .$$
\end{proposition}
\begin{proof}
First of all, by  proposition \ref{propsmooth}, $ X_{G\backslash 1} \backslash ( X_{G\backslash 1}\cap X_{G
\q 1})$ is smooth, so we may apply  proposition \ref{propdegen}.   Furthermore, we have 
\begin{eqnarray}  \gr(i) H_{prim}^n(X_{G\backslash 1}) & = &0 \quad  \hbox{ for  } \quad  0\leq i \leq 2 \ ,    \nonumber  \\
  \gr(i) H_{prim}^n(X_{G \q 1})  & = &0 \quad  \hbox{ for  }  \quad 0\leq i \leq 1  \ , \nonumber 
\end{eqnarray} 
by corollary \ref{corwd}, since $ 2h_{G\backslash 1} = 2(h_G -1) <  N_{G\backslash 1}.$ Thus
we have an isomorphism
\begin{equation} \label{drfp1}
\gr_F^{N-3} H_{dR}^{N-1} (\PP^{N-1} \backslash X_G)   \cong \gr_F^{N-3} H_{dR}^{N-2} (\PP^{N-2} \backslash (X_{G\backslash 1} \cup X_{G \q 1} ))  
\end{equation} 
and proposition \ref{propdegen} implies that 
$$ \Big[ {\Omega_{N-1} \over \Psi_G^2 }\Big] =0  \hbox{ in LHS if and only if  } 
\Big[ {\Omega_{N-2} \over \Psi_G^1 \Psi_{G,1} }\Big]=0 \hbox{ in RHS}\ .$$
For the next reduction, we   apply  proposition \ref{propsecondDRstep}, which implies that
 \begin{eqnarray}\label{drfp2}
 \hbox{if }  &\Big[ {\Omega_{N-2} \over \Psi_G^1 \Psi_{G,1}  }\Big]   \in    \gr_F^{N-3} H_{dR}^{N-2} (\PP^{N-2} \backslash   (X_{G\backslash 1} \cup X_{G \q 1} )) & \hbox{ vanishesÊ} \\
\hbox{ then }   & \Big[ {\Omega_{N-3} \over (\Psi_G^{1,2})^2 }\Big]    \in   \gr_F^{N-4}   H_{dR}^{N-3} (\PP^{N-3} \backslash V(\Psi_G^{1,2}))  & \hbox{ vanishes} \nonumber 
  \end{eqnarray}
 Now we apply proposition \ref{propdegen} to the hypersurface $V(\Psi_G^{1,2}) \subset \PP^{N-3}$. 
 By contraction-deletion, $\Psi_G^{1,2} = \Psi_G^{13,23} \alpha_3 + \Psi^{1,2}_{G,3}$, so the 
  smoothness assumption holds by corollary  \ref{corsmooth}.  Furthermore, we have 
  \begin{eqnarray}  \gr(i) H_{prim}^n(\cV(\Psi^{13,23}_{G})) & = &0 \quad  \hbox{ for  } \quad  0\leq i \leq 1 \ ,   \nonumber   \\
  \gr(i) H_{prim}^n(\cV(\Psi^{1,2}_{G,3 }))  & = &0 \quad  \hbox{ for  }  \quad i=0  \ , \nonumber 
\end{eqnarray} 
The first line   follows from the identity $\Psi_G^{13,23}=\Psi_{G\backslash 1, 3 \q 2}$ (see the comments following $(\ref{b9} )$)  and  equation $(\ref{PsiGVanishestoorder2})$;
the second line follows from the Chevalley-Warning theorem \ref{t20}, since $\deg \, ( \Psi^{1,2}_{G,3}) = h_{G} -1 \leq N_G -4$. 
  This gives \begin{equation} \label{drfp3}
\gr_F^{N-4} H_{dR}^{N-3} (\PP^{N-3} \backslash V(\Psi_G^{1,2}))  \cong   \gr_F^{N-4} H_{dR}^{N-4}(\PP^{N-4} \backslash \cV(\Psi_G^{13,23}, \Psi^{1,2}_{G,3}))   \qquad 
\end{equation}
and proposition \ref{propdegen} implies that 
$$ \Big[ {\Omega_{N-3} \over (\Psi_G^{1,2})^2 }\Big] =0   \hbox{    in LHS  if and only if }  \Big[ {\Omega_{N-4} \over \Psi^{13,23}_G \Psi^{1,2}_{G,3} }\Big] =0 \hbox{ in RHS}\ . 
$$
The result follows on combining $(\ref{drfp1})$,  $(\ref{drfp2})$, and $(\ref{drfp3})$.
\end{proof}
Suppose that $G$ satisfies the conditions of the previous proposition, and let $D_4,\ldots, D_k$ be a sequence
of denominators obtained by reducing out  the edges $1,2, \ldots, k$, where $1,2,3$ form a 3-valent vertex.

\begin{theorem} 
If the cohomology class
$$\Big[ {\Omega_{N-1} \over \Psi_G^2 } \Big]  \in \gr_F^{N- 3} H_{dR}^{N-1}(\PP^{N-1} \backslash X_G) $$
vanishes, then so does
$$ \Big[{ \Omega_{N-k-1} \over  D_k} \Big] \in 
   \gr_F^{N- k-1}   H_{dR}^{N-k-1}(\PP^{N-k-1} \backslash \cV(D_k)) $$
   \end{theorem}

\begin{proof}
The theorem follows immediately from the previous proposition to perform the first 3 reductions, followed by successive application of the
generic reduction step for differential forms (proposition \ref{propgeneric}).
\end{proof}

\begin{corollary} Let $G$ be as above,  and suppose that $G$ is denominator reducible (and non-weight drop). Then the vector-space
\begin{equation}Ê\gr_F^{N- 3} H_{dR}^{N-1}(\PP^{N-1} \backslash X_G) \label{grfcoreq}
\end{equation} 
is one-dimensional, spanned by the class of the Feynman differential form
$$\Big[ {\Omega_{N-1} \over \Psi_G^2 } \Big]\ .$$
\end{corollary}

\begin{proof}
Compare the proof of corollary \ref{corwd}, which yields the  one-dimensionality  of $(\ref{grfcoreq})$ by localization.
It is enough to show, by the previous theorem, that the final stage of the denominator reduction is non-zero. 
After a suitable change of coordinates,
$ {\Omega_{1} \over D_{N-2} }  =  {xdy-ydx \over xy}$ which is clearly non-zero in  $\gr^1_F H_{dR}^1(\PP^1 \backslash \{0, \infty\})$.
\end{proof}

\begin{remark} \label{remchangefunctor} All the results of this section are also valid for the weight filtration, if we replace
$\gr^k_F$ with $\gr^W_{2k}$ throughout. One can also replace $\gr^p_F$ with the  bigraded functor $\gr^{p,q}$, and the proofs are clearly unchanged, since we only require exactness of the functor
and vanishing results  for the Hodge numbers $(\ref{Xhodgenos})$.
\end{remark}
\section {On the eight-loop counter-example}
In the previous sections we computed the first graded piece
$\gr_{min}^W H^{N_G-2}(X_G)$ for any denominator-reducible graph $G$.  In this section we study the corresponding problem for a non denominator-reducible graph.
The first non-trivial counter-example  to
Kontsevich's conjecture on the number of rational points of graph
hypersurfaces was  given  explicitly in \cite{BrSch}. It is the primitive  overall log-divergent graph  $G_8$  with 9
vertices $1,\ldots,9$ and edges 
\begin{equation}
34,14,13,12,27,25,58,78,89,59,49,47,35,36,67,69,
\end{equation}
where $ij$ denotes an edge connecting vertices $i$ and $j$. This graph is 
depicted in  Figure 8 of \cite{BrSch}. It was proved that the point-counting function for this graph is given by a modular form, but it remains to show that the non-Tate
part of the cohomology occurs in middle degree. 

First of all we show that  
$$\gr^{2,4} H^{N_G-2}(X_{G_8})\neq 0\ ,$$
where $\gr^{2,4}$ denotes the part of Hodge type $(2,4$), which clearly  implies that  $H^{N_G-2}(X_{G_8})$ is not mixed Tate. Then, we shall compute the framing given by the Feynman differential form in the de Rham
cohomology of $\PP^{N_G-2} \backslash X_{G_8}$.

\subsection{Non-Tate cohomology}
The denominator reduction of the graph $G_8$ was computed in lemmas 55 and 56 in \cite{BrSch}. With the above ordering of the edges,
the denominators $D_0, \ldots, D_{11}$ are defined, and $D_{11}$ is given explicitly by 
\begin{equation}\label{c20}
  \pm D_{11}=
   \Psi^{15,78}_{A\backslash 11}\Psi_{B\q 11}
   -\Psi^{15,78}_{A\q11}\Psi_{B\backslash 11}
\end{equation}
where $A,B$ are the minors of $G_8$ defined by
 $$A=G_8\backslash\{2,3,5,10\}\q \{4,6,9\} \hbox{  and  } B=G_8\backslash\{2,3,5,7,8\}\q\{1,4,6,9,10\}. $$
  By  theorem \ref{thmcohomDR}, we immediately deduce that:

\begin{proposition} \label{propG_8D11}
For the graph $G_8$  consider
$\cV(D_{11})\subset\PP^4$. One gets
\begin{equation}\label{c21}
    gr^i_F H^{14}(X_{G_8})
    \cong
    gr^{i-2}_F H^{3}(\cV(D_{11})),\;\;
    i\leq 2.
\end{equation}
\end{proposition}
\noindent
The subvariety $\cV(D_{11},\alpha_{16})$ is again reducible
in the sense of Proposition \ref{p5}.
\begin{lemma}\label{l20}
One has $\gr^{0,2} H^{m}(\cV(D_{11},\alpha_{16}))=0$ for
any $m$.
\end{lemma}
\begin{proof}
After setting $\alpha_{16}=0$ into (\ref{c20}) we obtain
\begin{equation}\label{c22}
    D_{11}|_{\alpha_{16}=0}=-\alpha_{14}\alpha_{15}P
\end{equation}
where 
$P  = \alpha_{12}(\alpha_{12}+\alpha_{13}+\alpha_{15}) \alpha_{14}+\alpha_{13}(\alpha_{15}+\alpha_{12})(\alpha_{12}+\alpha_{13})$, which is  of degree $1$ in 
$\alpha_{14}$ and $\alpha_{15}$.
Using Proposition \ref{p5} twice with respect to $\alpha_{14}, \alpha_{15}$, we get
\begin{equation}
\gr^0_F H^m_{prim}(\cV(D_{11},\alpha_{16}))\cong \gr^0_F
H^{m-2}_{prim}(\cV(P, \alpha_{14}, \alpha_{15})))
\end{equation}
for all $m$. By inspection,  $\cV(P, \alpha_{14}, \alpha_{15}))$ is a union of three points in $\PP^1$ and therefore of Tate type.  Thus 
$\gr^{0,2} H^m(\cV(D_{11}, \alpha_{16}))$ vanishes for all $m$.
\end{proof}
 The localization sequence for $\cV(D_{11},\alpha_{16})\subset\cV(D_{11})$ reads
\begin{multline}
  \rightarrow H^{2}_{prim}(\cV(D_{11},\alpha_{16}))\rightarrow
  H^{3}_c(\cV(D_{11})\backslash\cV(D_{11},\alpha_{16}))\rightarrow\\
  H^{3}_{prim}(\cV(D_{11}))\rightarrow
  H^{3}_{prim}(\cV(D_{11},\alpha_{16}))\rightarrow 
\end{multline}
After applying  the functor $\gr^{0,2}$ to this sequence, the above lemma gives
\begin{equation}\label{c31}
   \gr^{0,2} H^{3}_{prim}(D_{11})\cong
  \gr^{0,2} H^{3}_c(D_{11}|_{\alpha_{16}\neq 0}).
\end{equation}

Let $\widehat{D}$ be $D^{11}_{G_8}|_{\alpha_{16}=1}$, and consider the corresponding affine scheme
 $\cV(\widehat{D}) \subset  \AAA^4 $, where $ \AAA^4\subset\PP^4$ is the open subset   $\alpha_{16}\neq 0$. Then $
H_c^{m}(\cV(D_{11})\backslash\cV(D_{11},\alpha_{16}))$ is equal to 
$H_c^{m}(\cV(\widehat{D}))$, for any $m$. Following the strategy in
\cite{BrSch}, we scale $\alpha_{12}$ and $\alpha_{13}$ by the polynomial
$\Psi_{\gamma} =\alpha_{14} \alpha_{15} +\alpha_{14} \alpha_{16} +\alpha_{15} \alpha_{16} $. This  transforms the polynomial $\widehat{D}$  into another  polynomial
called  $\widetilde{D}$ in \cite{BrSch}. 
This change of variables is an isomorphism  on the complement of $\cV(\Psi_{\gamma})$, giving
\begin{equation}\label{c34}
H^3_c(\cV(\widehat{D})\backslash \cV(\widehat{D},\Psi_{\gamma})) \cong
H^3_c(\cV(\widetilde{D})\backslash \cV(\widetilde{D},\Psi_{\gamma}))
\end{equation}
The left and right parts fit into the two exact sequences
\begin{equation} \nonumber
\rightarrow H^2_c(\cV(\widehat{D},\Psi_{\gamma}))\rightarrow
H^{3}_c(\cV(\widehat{D})\backslash \cV(\widehat{D},\Psi_{\gamma}))
\rightarrow H^3_c(\cV(\widehat{D}))\rightarrow
H^3_c(\cV(\widehat{D},\Psi_{\gamma}))\rightarrow
\end{equation}
\begin{equation} \nonumber
\rightarrow H^{2}_c(\cV(\widetilde{D})) \rightarrow
H^{2}_c(\cV(\widetilde{D},\Psi_{\gamma}))\rightarrow
H^3_c(\cV(\widetilde{D})\backslash \cV(\widetilde{D},\Psi_{\gamma}))
\rightarrow H^ 3_c(\cV(\widetilde{D}))\rightarrow
\end{equation}
\begin{lemma}\label{l21}
For $\cV(\widetilde{D})$ and $\cV(\widehat{D},\Psi_{\gamma})$ we have
$\gr^{0,2} H^m_c=0$ for all $m$.
\end{lemma}
\begin{proof}
 Let $Z$ be one these two  (affine) varieties.  Let
$\bar{Z}$ be its compactification obtained by homogenizing with respect to $\alpha_{16}$. An exact sequence of the form
\begin{equation}\label{c37}
\rightarrow H^m(\bar{Z}\cap\cV(\alpha_{16}))\rightarrow
H^m_c(Z)\rightarrow H^m(\bar{Z})\rightarrow
H^m(\bar{Z}\cap\cV(\alpha_{16}))\rightarrow,
\end{equation}
shows that it is enough to prove the statement of the lemma for 
$\bar{Z}$ and $\bar{Z} \cap \cV(\alpha_{16})$. 
As explained in \cite{BrSch},  lemma 59,  for the case $Z=\cV(\widetilde{D})$, $\widetilde{D}$ is of degree one in
$\alpha_{14}$ and $\alpha_{15}$, while in the second case
$Z=\cV(\widehat{D},\Psi_{\gamma})$ is a union of intersections of
hypersurfaces of degree at most 2 which are linear in every variable. Both
cases can be treated by Proposition \ref{p5}. The defining
equations of $Z\cap\cV(\alpha_{16})$ are even easier. In all cases the cohomology
$H^m$ has  no $\gr^{0,2}$ pieces. \end{proof}
By (\ref{c34}), Lemma \ref{l21}, and the two sequences preceding it,
we have
\begin{equation}\label{c42}
\gr^{0,2} H^3_c(\cV(\widehat{D}))\cong \gr^{0,2}
H_c^{2}(\cV(\widetilde{D}, \Psi_{\gamma}))\ .
\end{equation}
Since $\widetilde{D}$ is linear in $\alpha_{14}$, we can apply Proposition
\ref{p5} one more time to the pair $(\widetilde{D}, \Psi_{\gamma})$ with respect to $\alpha_{14}$, yielding:
\begin{equation}\label{c43}
\gr^{0,2}
H_c^{2}(\cV(\Psi_{\gamma},\widetilde{D}))\cong\gr^{0,2} H_c^2(\cV(P))\ ,
\end{equation}
where $P$ is the resultant:
\begin{multline}\label{c43.5}
P=\alpha_{12}+\alpha_{12}\alpha_{15}+\alpha_{13}\alpha^2_{12}+\alpha_{12}^2+\alpha_{13}\alpha_{12}+
\alpha_{15}\alpha_{13}\alpha_{12}\\
+\alpha_{13}^2\alpha_{15}+\alpha_{13}^2\alpha_{15}^2+\alpha_{13}^2\alpha_{15}\alpha_{12}+
\alpha_{15}^2\alpha_{13}\alpha_{12}.
\end{multline}
We introduce another change of variables
$\alpha_{13}\mapsto\alpha_{13}/\alpha_{15}+1$ and define $Q$ to be
the image of $P$ under this transformation. This can be handled in the same way as
above, giving an isomorphism
\begin{equation}\label{c44}
\gr^{0,2} H^2_c(P) \cong \gr^{0,2} H^2_c(Q).
\end{equation}
Now set $a=\alpha_{13}+1$, $b=\alpha_{12}+1$, $c=\alpha_{15}$. Then
$Q$ takes the form
\begin{equation}\label{c45}
J=a^2bc-ab-ac^2-ac+b^2c+ab^2+abc^2-abc,
\end{equation}
defining a singular surface in $\AAA^3$. Proposition \ref{propG_8D11} together with  (\ref{c42})
(\ref{c43}), (\ref{c44}), yields
\begin{equation}\label{c47}
    \gr^{2,4} H^{14}(X_{G_8})\cong \gr^{0,2} H^2_c(\cV(J)).
\end{equation}
Now let $T$ be the homogeneous polynomial
\begin{equation}\label{c48}
   T=b(a+c)(ac+bd)-ad(b+c)(c+d)
\end{equation}
satisfying $T|_{d=1}=J$. The complement $\cV(T|_{d=0})$ is a union of
lines, so a localization sequence implies that
\begin{equation}\label{c49}
\gr^{0,2} H^2(\cV(T)) \cong \gr^{0,2} H_c^2(J).
\end{equation}
By \cite{BrSch}, \S7, $\cV(T)$ has six singular points. Blowing them up defines a K3 surface $Y$.
Since the Hodge numbers of a K3 satisfy $h^{0,2}=h^{2,0}=1$, and
since  blowing-up points only  adds  extra summands of  Tate type, we conclude that
$\gr^{0,2} H^2(Y)\cong \gr^{0,2}
H^2(\cV(T))$ is one-dimensional. By (\ref{c47}) and (\ref{c49}),  $\gr^{2,4} H^{14}(X_G)$ is one-dimensional.

\subsection{The differential form in the 8-loop counter-example}
We now wish to chase the Feynman differential form in this example in order to prove that it defines a non-zero framing which is not of Tate type.

We start with the Feynman form 
\begin{equation}
\omega_{G_8}=\frac{\Omega_{15}}{\Psi_{G_8}^2}
\end{equation}
and consider its class in $H^{15}_{dR}(\PP^{15}\backslash X_{G_8})$. By 
remark \ref{remchangefunctor}, we can apply the functor $\gr^{p,p-2}$ throughout the argument of \S \ref{sectRedForms} instead of the functor $\gr_F^p$. By the general denominator reduction for differential forms,
we immediately obtain that
$$ [\omega_{G_8}] \in \gr^{13,11} H^{15}_{dR}(\PP^{15}\backslash X_{G_8}) \hbox{  vanishes}$$
 \begin{equation}\label{c63}
\hbox{ implies that }  \Big[\frac{\Omega_{4}}{D_{11}}\Big]\in \gr^{4,2} H^{4}_{dR}(\PP^{4}\backslash \cV(D_{11})) \hbox{ vanishes. } 
\end{equation}
It is enough to show that this latter class is non-zero.
Now,  following  the argument  in the paragraph after lemma \ref{l20}, consider the restriction of the form $\frac{\Omega_{4}}{D_{11}}$ to affine space $\alpha_{16}\neq 0$  
\begin{equation}
\beta_1=\frac{d\alpha_{12}\wedge d\alpha_{13}\wedge d\alpha_{14}\wedge d\alpha_{15}}{\widehat{D}}
\end{equation}    
Once again, nonvanishing of $[\beta_1]$ in $\gr^{4,2} H^{4}(\AAA^4\backslash\cV(\widehat{D}))$ implies  nonvanishing of $[\frac{\Omega_{4}}{D^{11}_G}]$ in $\gr^{4,2} H^{4}_{dR}(\PP^{4}\backslash \cV(D^{11}_G))$. 
After performing the change of the variables immediately preceding (\ref{c34}), we reduce to  proving that the class of the form 
\begin{equation}
\beta_2 = \frac{d\alpha'_{12}\wedge d\alpha'_{13}\wedge d\alpha_{14}\wedge d\alpha_{15}}{\widetilde{D}}
\end{equation}
(where  $\alpha'_{12}$ and $\alpha'_{13}$ are the rescaled versions of  $\alpha_{12}$ and $ \alpha_{13}$) is nonzero in $\gr^{4,2}H^{4}_{dR}(\AAA^{4}\backslash (\cV(\widetilde{D})\cup \cV(\Psi_{\gamma})))$. 
Let $V_1= \cV(\Psi_{\gamma})\backslash \cV(\widetilde{D}) \subset \AAA^3$. It is clearly smooth from the definition of $\Psi_{\gamma}$. We can therefore consider the 
  Gysin sequence 
\begin{equation}\label{c66}
H^2_{dR}(V_1)(-1)\rightarrow
H_{dR}^{4}(\AAA^3\backslash\cV(\widetilde{D}))
\rightarrow H^4_{dR}(\AAA^3\backslash (\cV(\widetilde{D})\cup \cV(\Psi_{\gamma})))\rightarrow
H^3_{dR}(V_1)(-1)
\end{equation}
Either by  the arguments in lemma \ref{l21}, or from the general bounds on  Hodge numbers, we have $\gr^{4,2} H^2_{dR}(V_1)(-1)=0$. It follows that  the map
$$\gr^{4,2} H_{dR}^{4}(\AAA^3\backslash\cV(\widetilde{D}))
\rightarrow \gr^{4,2} H^4_{dR}(\AAA^3\backslash (\cV(\widetilde{D})\cup \cV(\Psi_{\gamma})))$$ is injective, and it suffices to show that 
the class of $\beta_2$ in $\gr^{4,2}  H_{dR}^{4}(\AAA^3\backslash\cV(\widetilde{D}))$ is non-zero.
We can now  apply proposition \ref{propgeneric} one more time with respect to $\alpha_{14}$ (compare equation (\ref{c43})) and come to the form
\begin{equation}
\beta_3=\frac{d\alpha_{12}\wedge d\alpha_{13}\wedge d\alpha_{15}}{P}
\end{equation} 
where $P$ is defined in (\ref{c43.5}).
We  wish to show that its class  in $\gr^{3,1} H_{dR}^3(\AAA^3\backslash\cV(P))$ is non-zero.
We  next perform the change of variables given immediately preceding (\ref{c44}) and  come to  a class $[\beta_4]$  in $\gr^{3,1} H_{dR}^3(\AAA^3\backslash\cV(Q))$. The final change of the variables   
immediately following   (\ref{c44}) yields the form $\beta_5=\frac{d a\wedge d b\wedge d c}{J}$ on $\AAA^3\backslash \cV(J)$. Its cohomology class is the restriction to $d\neq 0$ of
\begin{equation} \label{lastform}
\Big[\frac{d a\wedge d b\wedge d c}{T}\Big]\in \gr^{3,1} H^3_{dR}(\PP^3\backslash \cV(T))
\end{equation}
A localization sequence again shows that it suffices to prove that this latter class is non-zero. After  desingularization (see the lines after (\ref{c49})), the  image  of $(\ref{lastform})$  in $\gr^{3,1} H^3_{dR}(\PP^3\backslash Y)$ 
maps   to the generator of $\gr^{2,0} H^2(Y)=H^0(Y,\Omega^2(Y))$ via the  residue map.  In particular, it is non-zero.
  The collection of  implications starting with (\ref{c63}) yields the non-vanishing of $[\omega_{G_8}] \in \gr^{13,11} H_{dR}^{15}(\PP^{15}\backslash X_{G_8})$. From the computations in the previous section and duality, we know that 
  this space is one-dimensional, so we conclude that 
   $$\gr^{13,11} H_{dR}^{15}(\PP^{15}\backslash X_{G_8}) \hbox{ is spanned by the class of } [\omega_{G_8}]\ .$$
Thus the non-Tate contribution to the cohomology of the graph hypersurface comes precisely from the class of the Feynman differential form. The period cannot therefore factorize (via some suitable notion of framed equivalence classes of motives, or motivic periods) through a category of mixed Tate motives.

\end{document}